\documentclass{article}
\usepackage{amsmath}
\usepackage{amsfonts, amssymb, amsthm}
\usepackage{url}
\usepackage{enumerate}
\usepackage{enumitem}
\usepackage{mathtools}
\usepackage{xcolor}
\usepackage{selinput}
\usepackage{soul}

\usepackage[TS1,T1]{fontenc}
\usepackage[osf]{mathpazo}
\usepackage{fbb}

\newtheorem{Theorem}{Theorem}[section]
\newtheorem{Proposition}[Theorem]{Proposition}
\newtheorem{Lemma}[Theorem]{Lemma}

\newtheorem{Assumption}[Theorem]{Assumption}
\newtheorem{Definition}[Theorem]{Definition}

\newtheorem{Remark}[Theorem]{Remark}
\numberwithin{equation}{section}

\newcommand{\Uad}{\mathcal{U}}

\newcommand{\dx}{\,\mathrm{d}x}

\usepackage{cite}

\usepackage{hyperref}

\title{Finite element error analysis of affine
optimal control problems\thanks{This work was supported by the FWF grants P-31400-N32 and I4571-N.}}
\author{Nicolai Jork\thanks{Institute of Statistics and Mathematical Methods in Economics, Vienna University of Technology, Vienna, Austria.}}
\date{}
\begin{document}

\maketitle
%
\begin{abstract}
This paper is concerned with error estimates for the numerical approximation for affine optimal control problems subject to semilinear elliptic PDEs. To investigate the error estimates, we focus on local minimizers that satisfy certain local growth conditions. The local growth conditions we consider in this paper appeared recently in the context of solution stability and contain the joint growth of the first and second variation of the objective functional.
These growth conditions are especially meaningful for affine control constrained optimal control problems because the first variation can satisfy a local growth, which is not the case for unconstrained problems. The main results of this paper are the achievement of error estimates for the numerical approximations generated by a finite element scheme with piecewise constant controls or a variational discretization scheme. 
Even though the growth conditions considered are weaker than those appearing in the recent literature on finite element error estimates for affine problems, this paper substantially improves the existing error estimates for both the optimal controls and the states when a Hölder-type growth is assumed.
\end{abstract}
\section{Introduction}
Affine optimal control problems, by which we mean problems where the controls appear at most in an affine way in the objective functional and the constraining equation, are a relatively recent subject of study, especially when PDE constraints are considered. For the analysis of affine optimal control problems subject to ODE constraints, we refer to the papers \cite{DV2021,F2003,MO2004, MO1998,   OM2005, OM2007, OV2020, PSV2019, PSV2018,QV2013} which contain results related to sufficient second-order conditions and the metric regularity and stability of the optimal control problems, especially for those with bang-bang structure of the optimal controls. An application of the regularity and stability investigations is for error estimates for the discretized problem, for instance, the Euler discretization, which can be found in \cite{OV2020}.
For PDE-constrained problems, the earliest work related to affine problems is, to the best knowledge of the author, the work \cite{Casas2012}, which was extended to problems with different constraining PDEs or objective functionals, for instance, in \cite{B2014,CT2014,CT2016,CWW2017,CRT2015,CMR2019}. Recently, the study of affine PDE-constrained optimal control problems was done in the works \cite{CDJ2023, DJV2022b,DJV2022} under assumptions resembling partially the assumptions that appeared in the context of ODE-optimal control in \cite{OV2020}.
A typical example of an affine problem is the tracking-type objective functional, which is a common type of objective functional used in many applied situations, including engineering, finance, and, more recently, machine learning.
Often, a so-called Tikhonov regularization term, a quadratic term with respect to the controls, is added to the objective functional. This is mainly for two reasons. First, in some situations, it is of interest to penalize the control cost. Second, incorporating the Tikhonov regularization term has some significant implications for analyzing the control problem. One of which is that by adding such a term, under mild additional assumptions, a quadratic growth of the second variation of the objective functional can be guaranteed, making the problem coercive. This then has many implications in the analysis of the problem, such as the study of error estimates for the numerical approximation. On the other hand, this comes with the price that adding such a term represents a distortion of the original problem; the optimal controls and states of the regularized problem can have substantially different structures. For instance, the bang-bang property of optimal controls can be expected in affine problems but not in regularized ones. To compensate for the missing coercivity due to the absence of a Tikhonov regularization term, the analysis of affine (unregularized) optimal control problems, in general, builds on certain assumptions related to the growth of the objective functional at local minimizers. In this paper, to study error estimates for the numerical approximation, we rely on assumptions that were recently studied in the context of strong metric subregularity and solution stability. These growth conditions encompass the local joint growth of the first and second variations and are weaker than the growth of the objective functional satisfied by Tikhonov regularized problems and the usual assumptions made in the study of affine problems. Using these assumptions, which we will specify below, we consider error estimates for the numerical approximation generated by a finite element scheme with piecewise constant controls and a variational discretization scheme. The analysis is motivated by the works \cite{CM2021,  CWW2018, DH2012,CMR2019} on which the results of this paper build. Still, in comparison to the results therein, the error estimates for the optimal controls and states are under the assumptions introduced in \cite{CDJ2023, DJV2022}, which are weaker than the ones assumed in  \cite{CM2021, CWW2018, DH2012,CMR2019}. For a detailed comparison in case of a parabolic constraining PDE, we refer to \cite{CM2020, DJV2022b}, for the convenience of the reader, we provide a short discussion at the end of Section \ref{CProblem}. Among others, we utilize the following assumption: given a reference optimal control $\bar u$ and a number $\gamma \in (0,1]$, there exist positive constants $\alpha$ and $c$ such that
\begin{equation}\label{cond1}
    J(u)-J(\bar u)\geq c\|u-\bar u\|_{L^1(\Omega)}^{1+\frac{1}{\gamma}} \text{ for all feasible controls } u \text{ with } \|u-\bar u\|_{L^1(\Omega)}<\alpha.
\end{equation}
Conditions of the type \eqref{cond1} arise naturally in characterizing strict bang-bang optimal controls. They also appear due to sufficient second-order optimality conditions and the structural assumption on the adjoint state, see \cite{CWW2018}. A slightly stronger assumption that implies \eqref{cond1} was first considered in \cite{OV2020} for affine ODE optimal control problems and \cite{DJV2022} for PDE optimal control problems. Recently, \eqref{cond1} appeared in \cite{M2022,M2020} in the context of eigenvalue optimization problems. There, it was shown that for a certain type of eigenvalue optimization problem, condition \eqref{cond1} is implied by a growth of the second-order shape derivatives.

In this paper, we do not explicitly consider a sparsity-promoting term in the objective functional as it is done, for instance, in \cite{CM2021}. Still, the proofs can be easily adapted to include such a term following the arguments in \cite{CM2021}. Also, we expect the approach presented in the paper to apply to the situation of having a semilinear elliptic non-monotone and non-coercive state equation as in \cite{CMR2020} using the results of \cite{CMR2021b}.

To the author's best knowledge, the assumptions considered in this paper are the weakest so far that still allow error estimates for the numerical approximation for problems where the control appears at most in an affine way in the objective functional, and we expect the approach discussed in this paper to be easily adaptable to optimal control problems constrained by various other PDEs. It may also be feasible to achieve error estimates for the numerical approximation for a $2$-dimensional Neumann boundary control problem. But we postpone the analysis to future work.
\bigskip

Let us list the novelties in the paper. 
In Proposition \ref{mixedbangbang}, we answer a question raised in \cite{CDJ2023} on the structure of optimal controls satisfying one of the assumptions introduced in \cite{CDJ2023}. This result allows us to apply the assumption for the investigation of error estimates for the numerical approximation later on. Under conditions similar to the one introduced in \cite{CDJ2023} in the context of solution stability and conditions \eqref{cond1}, we derive error estimates for a finite element discretization scheme with piecewise constant controls in Theorem \ref{pointest}. In the first part of the proof of the main theorem, Theorem \ref{pointest}, we argue similarly as in the first steps of the proof of \cite[Theorem 7]{CM2021}. In contrast to the proof in \cite[Theorem 7]{CM2021}, we employ arguments centered around the linearized state. This allows us to improve the error estimates for the optimal controls for $\gamma\in (0,1)$, from $\|\bar u_h -\bar u\|_{L^1(\Omega_h)}\leq ch^{\gamma^2}$ to $\|\bar u_h -\bar u\|_{L^1(\Omega_h)}\leq ch^{\gamma}$ (and similarly for the states). Under additional assumptions on the structure of the level set of the switching function, which is given by the adjoint state for particular tracking type problems, we can further improve the error estimate from $\|\bar u_h -\bar u\|_{L^1(\Omega_h)}\leq ch^{\gamma}$ to $\|\bar u_h -\bar u\|_{L^1(\Omega_h)}\leq ch^{\frac{2\gamma}{\gamma+1}}$ in Theorem \ref{pointest2}.
Using the assumptions of \cite{DJV2022, CDJ2023}, we prove error estimates for a variational discretization scheme in Theorem \ref{vardiscest} and discuss a relationship with solution stability afterward.
\bigskip

The paper is structured as follows: In the remainder of this section, we state the main assumptions that hold throughout the paper and state some additional remarks on the notation. In Section \ref{Auxpde}, we collect results on the involved PDEs, and in Section \ref{CProblem}, the optimal control problem is discussed. In Section \ref{discmee}, we define the discretization schemes and prove error estimates.
\bigskip 

Let $\Omega\subset\mathbb R^n$, $n\in \{2,3\}$, be a bounded domain with $C^{1,1}$-boundary. {Given constants $u_a, u_b\in \mathbb R$ such that $ u_{a} < u_{b} $}, define the set of feasible controls by
\begin{equation}
		\Uad := \{u \in L^\infty(\Omega) \vert \ u_a \le u(x) \le u_b\ \text{ for a.a. } x \in \Omega\}
	\label{const}
\end{equation}
and consider the optimal control problem
	\begin{equation}
 		  \ \min_{u \in \Uad}\Big\{ J(u) := \int_\Omega L(x,y(x),u(x))\dx\Big\},
	\label{ocp1}
	\end{equation}
subject to
	\begin{equation}
		\left\{ \begin{array}{lll}
		{A}y+f(\cdot,y) =u\  &\text{ in }\ \Omega,\\
		y=0\ &\text{ on } \Gamma.\\
	\end{array} \right.
	\label{see1}
	\end{equation}

Denote by $y_u$ the unique solution of the state equation that corresponds to the control $u$. The objective integrand $L$ appearing in \eqref{ocp1} satisfies additional smoothness conditions, given below in Assumption \ref{A3}.\\

\subsection{Main assumptions and notation}

The following assumptions, close to those in \cite{CDJ2023,CM2020,CM2021,CWW2018},
are standing in all of the paper.

\begin{Assumption}\label{A1}
	The following statements are fulfilled.
	\begin{itemize}
		\item[(i)] The operator $ A :H^1_0(\Omega)\to H^{-1}(\Omega)$, is given by
		\[
			 A y:=-\sum_{i,j=1}^{n}\partial_{x_j}(a_{i,j}(x)\partial_{x_i} y),
		\]
where $a_{i,j}\in C^{0,1}(\bar \Omega)$. Further, the  $a_{i,j}$ satisfy the uniform ellipticity condition
		\[
			\exists \lambda_{A}>0:\ \lambda_{A}\vert\xi\vert^2\leq \sum_{i,j=1}^{n} a_{i,j}(x)\xi_{i}\xi_{j} \;\;
               \textrm{ for all } \xi\in \mathbb{R}^n \;\; 	 \textrm{and a.a.}\ x\in \Omega.
		\]
		
		\item[(ii)] \label{A2}
We assume that $f:\Omega \times \mathbb{R}\longrightarrow \mathbb{R}$ is a Carath\'eodory function of class $C^2$ with respect to the second variable satisfying:
\begin{align*}
\left\{\begin{array}{l}
f(\cdot,0)\in L^\infty(\Omega) \text{ and } \frac{\partial f}{\partial y}(x,y) \geq 0 \  \forall y \in \mathbb{R},
\vspace{1mm}\\
\displaystyle \forall M>0\ \exists C_{f,M}>0 \text{ s. t. }\left\vert\frac{\partial f}{\partial y}(x,y)\right\vert+
\left\vert\frac{\partial^2 f}{\partial y^2}(x,y)\right\vert \leq C_{f,M} \ \forall \vert y\vert \leq M,
\vspace{1mm}\\
\displaystyle \forall \rho > 0 \text{ and } \forall M>0\ \exists\ \varepsilon > 0 \text{ such that } \\\displaystyle
\left\vert\frac{\partial^2f}{\partial y^2}(x,y_2) - \frac{\partial^2f}{\partial y^2}(x,y_1)\right\vert < \rho \ \forall \vert y_1\vert, \vert y_2\vert \leq M \mbox{ with }  \vert y_2 - y_1\vert \le \varepsilon, \end{array}\right.
\end{align*}
for almost every $x \in \Omega$.
\end{itemize}
\end{Assumption}

\begin{Assumption}
The function $L:\Omega \times \mathbb{R}^2 \longrightarrow \mathbb{R}$ is Carath\'eodory and of class $C^2$ with respect to the second variable. In addition, we assume that
\begin{align*}
\left\{\begin{array}{l} L(x,y,u) = L_a(x,y) + L_b(x,y)u \ \text{ with } \ L_a(\cdot,0), L_b(\cdot,0) \in L^1(\Omega),\vspace{1mm}\\
\displaystyle \forall M > 0 \ \exists  C_{L,M} > 0 \text{ such that }\vspace{1mm}\\
\displaystyle \Big\vert\frac{\partial L}{\partial y}(x,y,u)\Big\vert+ \Big\vert\frac{\partial^2L}{\partial y^2}(x,y,u)\Big\vert \le C_{L,M}\ \forall \vert y\vert, \vert u\vert \le M, \vspace{1mm} \\
\displaystyle \forall \rho > 0 \text{ and } M > 0 \ \vspace{1mm} \exists \varepsilon>0 \text{ such that }\\
\displaystyle\left\vert\frac{\partial^2L}{\partial y^2}(x,y_2,u) - \frac{\partial^2L}{\partial y^2}(x,y_1,u)\right\vert < \rho\ \vert y_1\vert, \vert y_2\vert \leq M \mbox{ with }  \vert y_2 - y_1\vert \le \varepsilon, \end{array}\right.
\end{align*}
for almost every $x \in \Omega$.
\label{A3}
\end{Assumption}
In the paper, we denote by $c$ a positive constant that may change its value from line to line.


\section{Auxiliary results for the state equation}\label{Auxpde}
We collect properties of solutions to linear and semilinear elliptic PDEs. The results in this section are standard by now; we refer to \cite{CDJ2023, CM2021}. In \cite{CDJ2023}, the results are obtained for a non-monotone and non-coercive semilinear elliptic PDE. The PDE considered in this paper can be seen as a special case, and the results apply.
Let $a_0 \in L^\infty(\Omega)$ be a nonnegative function. We consider the properties of solutions to the linear equation 
\begin{equation}
    A z+a_0 z= v \text{ in }\Omega,\ z=0 \text{ on } \Gamma.
   \label{lineq}
\end{equation}
\begin{Theorem}\cite[Lemma 2.2]{CDJ2023}
    Let $v\in L^r(\Omega)$ with $r>n/2$. Then the linear equation \eqref{lineq} has a unique solution $z_v\in H^1_0(\Omega)\cap C(\bar \Omega)$. Further there exists a positive constant $C_r$ independent of $a_0$ and $v$ such that
\begin{equation}\|z_v\|_{H^1_0(\Omega)}+\|z_v\|_{C(\bar \Omega)}\leq C_r\| v\|_{L^r(\Omega)}.
\label{solubound}
\end{equation}
\label{estlincont}
\end{Theorem}
\begin{Lemma}\cite[Lemma 2.3]{CDJ2023}
Assume that $s \in [1,\frac{n}{n - 2})$, $s'$ is its conjugate, and let $a_0 \in L^\infty(\Omega)$ be a nonnegative function. Then, there exists a constant $C_{s'}$ independent of $a_0$ such that
\begin{equation}
\|{z_v}\|_{L^s(\Omega)} \le C_{s'}\|{v}\|_{L^1(\Omega)}, \quad \forall {v} \in H^{-1}(\Omega) \cap L^1(\Omega),
\label{E2.3}
\end{equation}
where ${z_v}$ satisfies the equation \eqref{lineq}, and $C_{s'}$ is given by \eqref{solubound} with $r = s'$.
\label{L2.2}
\end{Lemma}
For the semilinear state equation, we cite the following regularity result.
\begin{Theorem}\cite[Theorem 1]{CM2021}\label{contandregsem}
For every $u \in L^r(\Omega)$ with $r > n/2$ there exists a unique $y_u \in Y :=
H^1_0 (\Omega)\cap C(\bar  \Omega)$ solution of \eqref{see1}. Moreover, there exists a constant $T_r > 0$ independent of $u$ such that
\begin{equation*}
    \| y_u\|_{H^1_0(\Omega)}+\|y_u\|_{C(\bar\Omega)}\leq T_r(\|u\|_{L^r(\Omega)}+\|f(\cdot,0)\|_{L^\infty(\Omega)}).
\end{equation*}
If $u_k\rightharpoonup u$ weakly in $L^r(\Omega)$, then we have the strong convergence 
\begin{equation*}
    \| y_{u_k}-y_u\|_{H^1_0(\Omega)}+\|y_{u_k}-y_u\|_{C(\bar\Omega)}\to 0.
\end{equation*}
Further if $u\in L^\infty(\Omega)$, we have $y_u\in W^{2,r}(\Omega)$ for all $r<\infty$ and 
    \begin{equation*}
        \|y_u\|_{W^{2,r}(\Omega)}\leq M_0r\Big( \| u\|_{L^\infty(\Omega)}+\| f(\cdot,0)\|_{L^\infty(\Omega)}\Big)
    \end{equation*}
    for a positive constant $M_0$ independent of $u$ and $r$.
\end{Theorem}

For each $r>n/2$, we define the map $G_r:L^r(\Omega)\to H_0^1(\Omega)\cap C(\bar\Omega)$ by $G_r(u)=y_u$.

\begin{Theorem}\cite[Theorem 2.6]{CDJ2023}
		Let Assumption \ref{A1} hold. For every $r > \frac{n}{2}$ the map $G_r$ is of class $C^2$, and the first and second derivatives at $u\in L^r(\Omega)$ in the directions $v, v_1, v_2\in L^r(\Omega)$, denoted by $z_{u,v} = G'_r(u)v$ and $z_{u,v_1,v_2} = G''_r(u)(v_1,v_2)$, are the solutions of the equations 
\begin{align}
& Az+\frac{\partial f}{\partial y}(x,y_u)z = v \text{ in }\Omega,\ z=0\text{ on } \Gamma,\label{E2.9}\\
& Az+\frac{\partial f}{\partial y}(x,y_u)z = -\frac{\partial^2f}{\partial y^2}(x,y_u)z_{u,v_1}z_{u,v_2} \text{ in }\Omega, \ z=0\text{ on } \Gamma,\label{E2.10}
\end{align}
respectively.
\label{T2.2}
\end{Theorem}
\begin{Lemma}\cite[Lemma 2.7]{CDJ2023}
The following statements are fulfilled.
Suppose that $s \in [1,\frac{n}{n - 2})$. Then, there exist a constant $M_s$ depending on $s$ such that for every $u, \bar u \in \Uad$
\begin{align}
&\|y_u - y_{\bar u} - z_{\bar u,u - \bar u}\|_{L^{s}(\Omega)}\le M_s\|y_u-y_{\bar u}\|^2_{L^{2}(\Omega)}.
\label{E2.12}
\end{align}
There exists $\varepsilon >0$ such that for all $u,\bar u\in \mathcal{U}$ with $\|y_u-y_{\bar u}\|_{C(\bar \Omega)}<\varepsilon$ the following inequality is satisfied
\begin{align}\label{E2.14}
1/2\|y_u-y_{\bar u}\|_{C(\bar \Omega)} \leq \|z_{\bar u,u-\bar u}\|_{C(\bar \Omega)}\leq 3/2\|y_u-y_{\bar u}\|_{C(\bar \Omega)}.
\end{align}
\begin{align}\label{E2.14b}
1/2\|y_u-y_{\bar u}\|_{L^2(\Omega)} \leq \|z_{\bar u,u-\bar u}\|_{L^2(\Omega)}\leq 3/2\|y_u-y_{\bar u}\|_{L^2(\Omega)}.
\end{align}
\label{L2.3}
\end{Lemma}	

We need the following lemma, which is standard, proofs \eqref{equation:check}can be found in \cite[Theorem 8.30]{GT1983} or \cite[Theorem 4.2]{G1965} respectively. The estimate \eqref{equation:wellkn} is standard and also appears in \cite[Lemma 1]{CM2021}.

{
\begin{Lemma}\label{stsupdual}
Let $u, \bar u\in \mathcal U$. Then for $n<p$ and some positive constant $\check C_p$ independent of $u$ and $\bar u$ it holds
\begin{equation}\label{equation:wellkn}
    \| y_u-y_{\bar u}\|_{H^1_0(\Omega)}\leq  1/\lambda_A \|u-\bar u\|_{W^{-1,2}(\Omega)}.
\end{equation}
\begin{equation}\label{equation:check}
    \| y_u-y_{\bar u}\|_{C(\bar \Omega)}\leq \check C_p \|u-\bar u\|_{W^{-1,p}(\Omega)}.
\end{equation}
\end{Lemma}
}
\section{The optimal control problem}\label{CProblem}
The optimal control problem \eqref{const}-\eqref{ocp1} is well posed under Assumptions \ref{A1} and \ref{A3}. By the direct method of calculus of variations, we obtain the existence of at least one global minimizer, 
see \cite[Theorem 5.7]{Troltzsch2010}. In this section, we calculate the first and second variation of the objective functional, state the first-order necessary optimality conditions, and introduce the sufficient conditions for optimality.
\begin{Definition}
We say that $\bar u \in \mathcal U$ is an $L^r (\Omega)$-weak local minimum of problem \eqref{const}-\eqref{see1}, if there exists 
some positive $\varepsilon$ such that
\[
J( \bar u) \le J(u) \  \ \ \forall u \in \mathcal U \text{ with } \|u-\bar u\|_{L^r(\Omega)}\le \varepsilon.
\]
We say that $\bar u \in \mathcal U$ is a strong local minimum of \eqref{const}-\eqref{see1}, if there exists $\varepsilon >0$ such that
\[
J( \bar u) \le J(u) \ \ \ \forall u \in \mathcal U \text{ with } \|y_u-y_{\bar u}\|_{C(\bar \Omega)}\le \varepsilon.
\]
We say that $\bar u \in \mathcal U$ is a strict weak (strong) local minimum if the above inequalities are
strict for $u\ne \bar u $.
\end{Definition}
Strong local minimizers were first considered in \cite[Definition 1.6]{BBS2014}. For a discussion of these notions of optimality, we refer to \cite[Lemma 2.8]{CM2020}.

\begin{Theorem} \label{T3.1}
For every $r > \frac{n}{2}$, the functional $J:L^r(\Omega) \longrightarrow \mathbb{R}$ is of class $C^2$. Moreover, 
given $u, v, v_1, v_2 \in L^r(\Omega)$ we have
\begin{align*}
     J'(u)v& =\int_\Omega\Big[\frac{\partial L}{\partial y}(x,y_u,u)\Big]z_{u,v}+\Big[\frac{\partial L}{\partial u}(x,y_u,u)\Big]v\dx= \int_\Omega\Big[\varphi_u+{L_b(x,y_u)}\Big]v\dx,
\end{align*}
\begin{align*}
   J''(u)(v_1,v_2)& = \int_\Omega\Big[\frac{\partial^2L}{\partial y^2}(x,y_u,u) - \varphi_u\frac{\partial^2f}{\partial y^2}(x,y_u)\Big]
       z_{u,v_1}z_{u,v_2}\dx + \int_\Omega \Big[{\frac{\partial L_b}{\partial y}(x,y_u)}\Big](z_{u,v_1}v_2+z_{u,v_2}v_1)\dx.
\end{align*}
Here, $\varphi_u \in H^1_0(\Omega) \cap C(\bar \Omega)$ is the unique solution of the adjoint equation
\begin{equation}
      \left\{\begin{array}{l}\displaystyle A \varphi + \frac{\partial f}{\partial y}(x,y_u)\varphi=  
    \frac{\partial L}{\partial y}(x,y_u,u) \text{ in } \Omega,\\ \varphi = 0\text{ on } \partial\Omega.\end{array}\right.
\label{E3.6}
\end{equation}
\end{Theorem}
{Due to the standing assumptions of the paper, we can even infer that $\varphi_u\in W^{2,p}$, $p<\infty$. To obtain this regularity for a variational solution to \eqref{E3.6} the boundary must have regularity $C^{1,1}$, see \cite[Section 2]{G1985}}. 
\noindent We define the Hamiltonian 
$\Omega\times\mathbb R\times \mathbb R\times \mathbb R \ni (x,y,\varphi,u) \mapsto H(x,y,\varphi,u) \in \mathbb R$
by
\begin{align}\label{Hamil}
       H(x,y,\varphi,u):=L(x,y,u)+\varphi(u-f(x,y)).
\end{align}
The following local form of the Pontryagin type necessary optimality conditions for problem \eqref{const}-\eqref{ocp1} stated below,
is well known (see e.g. \cite{C1993Pont, CM2020, Troltzsch2010} and \cite[Theorem 4]{CM2021}).
 
\begin{Theorem}
If $\bar u$ is a weak or strong local minimizer for problem  \eqref{ocp1}-\eqref{const}, 
then there exist unique elements $\bar y, \bar \varphi\in H^1_0(\Omega)\cap L^\infty(\Omega)$ such that 
\begin{align}
& \left\{\begin{array}{l} A\bar y + f(x,\bar y) = \bar u \text{ in } \Omega,\\ \bar y = 
0\text{ on } \partial\Omega.\end{array}\right.
\label{E3.7}\\
&\left\{\begin{array}{l}\displaystyle A\bar \varphi =
 \frac {\partial H}{\partial y}(x,\bar y,\bar \varphi,\bar u) \text{ in } \Omega,\\ \bar \varphi = 0\text{ on } \partial\Omega.\end{array}\right.
\label{E3.8}\\
&\int_\Omega\frac {\partial H}{\partial u}(x,\bar y,\bar \varphi,\bar u)(u - \bar u)\dx \ge 0 \quad \forall u \in \Uad. \label{E3.9}
\end{align}
\label{pontryagin}
\end{Theorem}

\subsection{Sufficient assumption for local optimality}\label{asuffcond}
In this subsection, we discuss three assumptions of different strengths that all imply strict local optimality and appeared recently in the context of affine PDE-constrained optimal control problems in \cite{CDJ2023, DJV2022}. In what follows, {$(\bar u,\bar y,\bar \varphi)$ denotes a fixed triplet satisfying the first-order necessary optimality condition. To shorten the notation, we denote $\bar H_u(x):=\frac{\partial H}{\partial u}(x,\bar y(x),\bar \varphi(x),\bar u(x))$}. 
\begin{Assumption}\label{asuff2}
{Let $\gamma\in (2/(2+n),1]$ and $\beta\in \{1/2,1\}$ be given}. There exist positive constants $\kappa$ and $\alpha$ such that
\begin{equation}\label{growthcontstab}
J'(\bar u)(u-\bar u)+\beta J''(\bar u)(u-\bar u)^2\geq \kappa\|u-\bar u\|_{L^1(\Omega)}^{1+\frac{1}{\gamma}}
\end{equation}
for all $u \in \mathcal U$ with $\|u-\bar u\|_{L^1(\Omega)}<\alpha$.
\end{Assumption}
Assumption \ref{asuff2}$(\beta=1)$ was first considered in the context of elliptic PDE-constrained optimization in \cite{DJV2022}. If $\bar u$ satisfies the first-order optimality condition \eqref{E3.9}, Assumption \ref{asuff2}$(\beta=1)$ implies Assumption \ref{asuff2}$(\beta=1/2)$. If the second variation of the objective funcional at the control $\bar u$ is nonnegative, the cases $\beta \in \{1/2,1\}$ are equivalent. Indeed, the second variation can be negative at $\bar u$ for certain directions for box-constrained optimal control problems; see, for instance, \cite[Example 2]{DJV2022b}.
Further Assumption \ref{asuff2} implies the bang-bang structure of the control $\bar u$, see \cite[Proposition 4.1]{DJV2022}. {At this point, let us also remark that if the control $\bar u$ is bang-bang, then the convergence $u_k \rightharpoonup^\star \bar u$ in $L^\infty(\Omega)$ implies the strong convergence of $\{u_k\}_{k=1}^\infty$ to $\bar u$ in $L^1(\Omega)$, see \cite[Lemma 4.2]{DJV2022}. }

Let us consider the two assumptions on the optimal control problem introduced in \cite{CDJ2023}. Due to \eqref{E2.3}, they present a weakening of Assumption \ref{asuff2}. 
\begin{Assumption}\label{asuffstm}
{Let $\beta \in \{1/2,1\}$ be given}. There exist positive constants $\kappa$ and $\alpha$ with
\begin{equation}\label{growthasuffstm}
J'(\bar u)(u-\bar u)+\beta J''(\bar u)(u-\bar u)^2\geq \kappa\|z_{\bar u,u-\bar u}\|_{L^2(\Omega)}\|u-\bar u\|_{L^1(\Omega)}
\end{equation}
for all $u \in \mathcal U$ with $\|y_u-y_{\bar u}\|_{C(\bar \Omega)}<\alpha$.
\end{Assumption}

\begin{Assumption}\label{asuffst}
{Let $\beta\in\{1/2,1\}$ be given}. There exist positive constants $\kappa$ and $\alpha$ with
\begin{equation}\label{growthasuffst}
J'(\bar u)(u-\bar u)+\beta J''(\bar u)(u-\bar u)^2\geq \kappa\|z_{\bar u,u-\bar u}\|_{L^2(\Omega)}^{2}
\end{equation}
for all $u \in \mathcal U$ with $\|y_u-y_{\bar u}\|_{C(\bar \Omega)}<\alpha$.
\end{Assumption}

Assumption \ref{asuffst}, the weakest of the three assumptions and does not imply the bang-bang property of the optimal controls. Assumption \ref{asuffst} is especially interesting as it is the weakest assumption so far that allows for solution stability estimates of the optimal states, see \cite{CDJ2023}. The interest of Assumption \ref{asuffstm} stems from the fact that it is the weakest assumption so far that still allows for solution stability for the optimal controls, which is also discussed in \cite{CDJ2023}. Further, in \cite{CDJ2023}, it was conjectured that Assumption \ref{asuffstm} may also be satisfied by optimal controls that are not bang-bang. If $\frac{\partial L_b}{\partial y}=0$, we can answer this negatively in the following proposition.
\begin{Proposition}\label{mixedbangbang}
{Let Assumption \ref{asuffstm} be satisfied and $\frac{\partial L_b}{\partial y}=0$}. 
     Then, $\bar u$ is bang-bang.
\end{Proposition}
\begin{proof}
    Assume that $\bar u$ is not bang-bang and let it satisfy Assumption \ref{asuffstm}. Since $\bar u$ is not bang-bang, there exists a set of positive measure $E\subset \Omega$, such that $\bar H_u=0$ on $E$. Let $v_E$ denote a control with $v_E=\bar u$ on $\Omega\setminus E$ and $\|v_E-\bar u\|_{L^1(\Omega)}<\alpha$. Then the first variation in direction $v_E-\bar u$ is zero and by \eqref{growthasuffstm}, we find 
    \begin{equation}\label{mixh2}
        \beta J''(\bar u)(v_E-\bar u,v_E-\bar u)\geq \kappa\|z_{\bar u,v_E-\bar u}\|_{L^2(\Omega)}\|v_E-\bar u\|_{L^1(\Omega)}.
    \end{equation}
{By the affine structure of the optimal control problem, Assumption \ref{A1}(ii), Assumption \ref{A3}, the fact that $\bar y, \bar \varphi\in C(\bar \Omega)$, the calculations in Theorem \ref{T3.1} and the assumption that $\frac{\partial L_b}{\partial y}=0$},  we infer the existence of a positive constant $c$ independent of the control $v_E$ such that
\begin{equation}\label{mixh1}
    \beta J''(\bar u)(v_E-\bar u,v_E-\bar u)\leq \beta c\|z_{\bar u, v_E-\bar u}\|_{L^2(\Omega)}^2.
\end{equation}
Thus, using \eqref{mixh2} and \eqref{mixh1}, we conclude for all controls $v_E$ with $v_E=\bar u$ on $\Omega\setminus E$
\begin{equation}\label{mixh3}
    \|\bar u-v_E\|_{L^1(\Omega)}\leq \beta c/\kappa\|z_{\bar u, v_E-\bar u}\|_{L^2(\Omega)}.
\end{equation}
 Since $\bar u$ is not bang-bang, we can select an $\varepsilon>0$ and a set of positive measure {$\hat E$} such that $\bar u(x) \in [u_a(x)+\varepsilon,u_b(x)-\varepsilon ]$ for a.e. {$x\in \hat E$}. Now consider a sequence $\{v^k_{\varepsilon}\}_{k=1}^\infty$ with $v^k_{\varepsilon}\in \{-\varepsilon,\varepsilon\}$ a.e. on $\Omega$ and $v^k_{\varepsilon}\rightharpoonup^* 0$ in $L^\infty(\Omega)$ for $k\to \infty$. Finally, define a sequence $\{\bar v_\varepsilon^k\}_{k=1}^\infty$ by $\bar v_\varepsilon^k:=\bar u$ on {$\Omega\setminus\hat E$} and $\bar v_\varepsilon^k:=\bar u+v^k_\varepsilon$ on {$\hat E$}. It is clear that $\bar v_\varepsilon^k\rightharpoonup^* \bar u$ in $L^\infty(\Omega)$ and $\|\bar v_\varepsilon^k-\bar u\|_{L^1(\Omega)}=\varepsilon \vert \hat E \vert$ for all $k\in \mathbb N$. On the other hand, by Theorem \ref{contandregsem}, $\bar v_\varepsilon^k\rightharpoonup^* \bar u$ in $L^\infty(\Omega)$ implies $\|z_{\bar u, \bar v_\varepsilon^k-\bar u}\|_{L^2(\Omega)}\to 0$ as $k\to \infty$.  This contradicts \eqref{mixh3}.
\end{proof}
Consequently, the notion of strong or weak local minimizer is equivalent under Assumption~\ref{asuff2} and Assumption~\ref{asuffstm}.
\begin{Lemma}\label{equiconstat}
{Let $\gamma \in (0,1)$ and $\beta\in \{1/2,1\}$ be given}. It is equivalent:
\begin{enumerate}
\item There exist positive constants $\kappa$ and $\alpha$ such that
\begin{equation}\label{contasum}
J'(\bar u)(u-\bar u)+\beta J''(\bar u)(u-\bar u)^2\geq \kappa\| u-\bar u\|_{L^1(\Omega)}^{1+\frac{1}{\gamma}}
\end{equation}
for all $u \in \mathcal U$ with $\|u-\bar u\|_{L^1(\Omega)}<\alpha$.
\item There exist positive constants $c$ and $\alpha$ such that \eqref{contasum} holds
for all $u \in \mathcal U$\\ with $\|y_u-y_{\bar u}\|_{C(\bar \Omega)}<\alpha$.
\end{enumerate}
Further, if the objective integrand satisfies $\frac{\partial L_b}{\partial y}=0$ it is equivalent
\begin{enumerate}
\item There exist positive constants $\kappa$ and $\alpha$ such that
\begin{equation}\label{contasumm}
J'(\bar u)(u-\bar u)+\beta J''(\bar u)(u-\bar u)^2\geq \kappa\| u-\bar u\|_{L^1(\Omega)}\|z_{\bar u,u-\bar u}\|_{L^2(\Omega)}
\end{equation}
for all $u \in \mathcal U$ with $\|u-\bar u\|_{L^1(\Omega)}<\alpha$.
\item There exist positive constants $\kappa$ and $\alpha$ such that \eqref{contasumm} holds
for all $u \in \mathcal U$\\ with $\|y_u-y_{\bar u}\|_{C(\bar \Omega)}<\alpha$.
\end{enumerate}
\end{Lemma}
\begin{proof}
The statement of Lemma \ref{equiconstat} for Assumption \ref{asuff2} with $\gamma=1$ was proven in  \cite[Proposition 5.2]{CDJ2023}. The proof relies on the fact that  \eqref{contasum} implies the bang-bang structure of $\bar u$. But if $\gamma \in (0,1)$, \eqref{contasum} still implies the bang-bang structure and the arguments in \cite[Proposition 5.2]{CDJ2023} hold true for $\gamma \in (0,1)$. By Proposition \ref{mixedbangbang}, the Assumption \ref{asuffstm} implies the control $\bar u$ to be bang-bang, thus the results can be obtained by the arguments as in \cite[Proposition 5.2]{CDJ2023}.
\end{proof}

The next lemmas are needed for the estimations later on. Their well-known statement was proven for objective functionals with varying generality for the case $\gamma=1$, \cite[Lemma 11]{DJV2022}. The proof for $\gamma\in (2/(2+n),1)$ follows by the same arguments. For Lemma \ref{estfv2} below, see for instance \cite[Lemma 2.7]{Casas2012}.
\begin{Lemma}\label{estfv}
Given $\gamma\in (2/(2+n),1]$ and $\bar u, u\in \mathcal U$. Define $u_\theta:=\bar u+\theta (u-\bar u) $ for some {measurable function $\theta$ with $0\leq \theta(x)\leq 1$}. For all $\epsilon >0$ there exists $\delta>0$ such that
\begin{equation*}
\Big \vert J''(\bar u)(u-\bar u)^2 -J''( u_\theta)(u-\bar u)^2\Big \vert\leq \epsilon \|u-\bar u\|_{L^1(\Omega)}^{1+\frac{1}{\gamma}}
\end{equation*}
for all $\|u-\bar u\|_{L^1(\Omega)}<\delta$.
\end{Lemma}

\begin{Lemma}\label{estfv2}
Given $\bar u, u\in \mathcal U$. Let $\frac{\partial^2 L}{\partial uy}=0$ and define $u_\theta:=\bar u+\theta (u-\bar u) $ for some {measurable function $\theta$ with $0\leq \theta(x)\leq 1$}. For all $\epsilon >0$ there exists $\delta>0$ such that
\begin{equation*}
\Big \vert J''(\bar u)(u-\bar u)^2 -J''( u_\theta)(u-\bar u)^2\Big \vert\leq \epsilon \|z_{\bar u,u-\bar u}\|_{L^2(\Omega)}^{2}
\end{equation*}
for all $\|y_u-y_{\bar u}\|_{C(\bar \Omega)}<\delta$.
\end{Lemma}
As a consequence of the lemmas \ref{estfv} and \ref{estfv2}, we obtain strict local optimality under Assumptions \ref{asuff2}, \ref{asuffstm} and \ref{asuffst}.
\begin{Theorem}\label{eqhalf}
Let $\bar u\in \mathcal U$ be given and  $\beta\in \{1/2,1\}$ in the assumptions reference below.
\begin{enumerate}
    \item Let Assumption \ref{asuff2} hold. Then there exists positive constants $\kappa$ and $\alpha$ such that
    \begin{equation}\label{coer1}
    J(u)-J(\bar u)\geq \kappa\|u-\bar u\|_{L^1(\Omega)}^{1+\frac{1}{\gamma}},
    \end{equation}
    for all $u\in \mathcal U$ with $\|u-\bar u\|_{L^1(\Omega)}<\alpha$.
        \item Let Assumption \ref{asuffst} hold for $\bar u\in \mathcal U$. There exist positive constants $\kappa$ and $\alpha$ such that
        \begin{equation}\label{statcoer}
        J(u)-J(\bar u)\geq \kappa\|z_{\bar u,u-\bar u}\|_{L^2(\Omega)}^{2}
        \end{equation}
        for all $u \in \mathcal U$ with $\|y_{\bar u}-y_u\|_{C(\bar \Omega)}<\alpha$.
       \item Let Assumption \ref{asuffstm} hold for $\bar u\in \mathcal U$. Then there exist positive constants $\kappa$ and $\alpha$ such that
    \begin{equation}\label{mixedvarcoer}
        J(u)-J(\bar u)\geq \kappa\|z_{\bar u,u-\bar u}\|_{L^2(\Omega)}\|\bar u-u\|_{L^1(\Omega)}
        \end{equation}
        for all $u \in \mathcal U$ with $\|y_{\bar u}-y_u\|_{C(\bar \Omega)}<\alpha$.
    \end{enumerate}
\end{Theorem}
\begin{proof}
    The statements with the growths \eqref{statcoer} and \eqref{mixedvarcoer} were proved in \cite{CDJ2023}. The statement for \eqref{coer1} follows by the same arguments using Lemma \ref{estfv}.
\end{proof}
\begin{Remark}
Assumption \ref{asuff2} with $\gamma\in (2/(n+2),1]$, together with Lemma \ref{estfv} is used to guarantee the existence of positive constants $\kappa$ and $\alpha$ such that
\begin{equation}\label{grwothrem}
J(u)-J(\bar u) \geq \kappa\|u-\bar u\|_{L^1(\Omega)}^{1+\frac{1}{\gamma}}, \ \text{for all } u\in \mathcal{U} \text{ with } \|u-\bar u\|_{L^1(\Omega)}<\alpha.
\end{equation}
If {$\frac{\partial^2 L}{\partial y\partial u}=\frac{\partial L_b}{\partial y}=0$} holds for the objective integrand, as a consequence of Lemma \ref{estfv2}, \eqref{grwothrem} can be obtained by considering Assumption \ref{asuff2}$(\beta=1/2,\ \gamma \in (0,1])$ together with Assumption \ref{asuffst}$(\beta=1/2)$.

To see this, let $\kappa$ and $\alpha$ be positive constants for that Assumption \ref{asuff2} and \ref{asuffst} are satisfied simultaneously. Then, if $\|\bar u_h-\bar u\|_{L^1(\Omega)}$ is sufficiently small, {applying Taylor's theorem, Assumption \ref{A3}, and Lemma \ref{estfv2}} yields 
\begin{align*}
    J(u)&=J(\bar u) +J'(\bar u)(u-\bar u)+1/2J''(\bar u_\theta)(u-\bar u)^2\geq J(\bar u) +1/2J'(\bar u)(u-\bar u)+1/4J''(\bar u)(u-\bar u)^2\\
    &+1/2J'(\bar u)(u-\bar u)+1/4J''(\bar u)(u-\bar u)^2-1/2\Big \vert J''(\bar u_\theta)(u-\bar u)^2-J''(\bar u)(u-\bar u)^2\Big\vert\\
    &\geq \kappa/2 \|u-\bar u\|_{L^1(\Omega)}^{1+\frac{1}{\gamma}}+\kappa/4 \|z_{\bar u,u-\bar u}\|_{L^2(\Omega)}^2.
\end{align*}
Thus, the constraint {$\gamma\in (2/(2+n),1]$} can be weakened to $\gamma \in (0,1]$ for the cost of making both, Assumption \ref{asuff2} and \ref{asuffst} at the same time.
\end{Remark}

\subsection{A short comparison with growth-related conditions in the literature}
We provide a short discussion of the relationship of Assumptions \ref{asuff2}, \ref{asuffstm} and \ref{asuffst} and the by now classical assumptions used for the analysis of affine PDE-constrained optimal control problems in the literature. By classical assumptions, we understand the ones considered, for instance, in \cite{Casas2012,CM2021, CMR2021b,CWW2018}.
For this, let us define cones appearing in affine PDE-constrained optimal control. 
\begin{Definition}
We consider the set
\begin{align}
\Big\{v\in L^2(\Omega)\Big \vert \  v\geq 0\text{ a.e. on } [\bar u =u_a]\text{ and } v\leq 0 \text{ a.e. on } [\bar u =u_b]\Big\}.
\label{sign}
\end{align}
Given $\tau>0$, we define the sets
\begin{align*}
D^{\tau}_{\bar u}&:=
\Big\{v\in L^2(\Omega)\Big \vert\   v\text{ satisfies }\eqref{sign}\text{ and } v(x)=
0\text{ if }\Big\vert  \frac{\partial \bar H}{\partial u}(x)\Big \vert   >\tau\Big\},\\
G^{\tau}_{\bar u}&:=
\Big\{v\in L^2(\Omega)\Big \vert\ v\text{ satisfies }\eqref{sign}\text{ and } J'(\bar u)(v)\leq \tau \|z_{\bar u,v} \|_{L^1(\Omega)}\Big\},\\
C^\tau_{\bar u}&:=D^{\tau}_{\bar u}\cap G^{\tau}_{\bar u}.
\end{align*}
Here, $\bar H$ denotes the Hamiltonian \eqref{Hamil} corresponding to the control $\bar u$, that is $\bar H(x):=H(x, \bar y (x),  \bar \varphi (x), \bar u(x))$.
\end{Definition}
Usually, the following two assumptions are made for the analysis of affine problems.
The first is the structural assumption on the switching function: There exist positive constants $c$ and $\gamma \in (0,1]$ such that 
\begin{equation}\label{struct}
    \vert \{ x\in \Omega \vert\ \vert \bar H_u\vert \leq \varepsilon\}\vert \leq c \varepsilon^{\gamma}.
\end{equation}
It is well known that this assumption implies for a possible different constant $c$ that
\begin{equation}\label{grfir}
    J'(\bar u)(u-\bar u)\geq c \|u-\bar u\|_{L^1(\Omega)}^{1+1/\gamma}\ \ \text{ for all } u\in \mathcal{U}.
\end{equation}
The second assumption is the so-called second-order sufficient condition
\begin{equation}\label{sosc}
    J''(\bar u)(u-\bar u)^2\geq c \| z_{\bar u,u-\bar u}\|_{L^2(\Omega)}^2 \ \text{ for all } u\in \mathcal{U} \text{ with } (u-\bar u)\in C^\tau_{\bar u}.
\end{equation}
Utilizing \eqref{struct} and \eqref{sosc}, error estimates for the numerical approximation are provided in \cite{CM2021}.

To compare these assumptions with the one used in this paper, we first notice that it is equivalent to consider Assumption \ref{asuff2} only for $u\in \mathcal U$ with $(u-\bar u)\in D^\tau_{\bar u}$, see \cite[Proposition 6.2]{DJV2022} for elliptic problems and for parabolic problems see \cite[Corollary 14]{CDJ2023}.
Further, we have the following theorem that relates Assumption \ref{asuff2} to \eqref{struct}, \eqref{grfir} and \eqref{sosc}.
\begin{Theorem}\label{compasuff}
Let $\frac{\partial L_b}{\partial y}=0$ and let there exist positive constants $c, k$ and $\alpha$ with $k<c$ such that
\begin{equation*}
    J'(\bar u)(u-\bar u)\geq c\|u-\bar u\|^{1+\frac{1}{\gamma}}_{L^1(\Omega)} \text{ for all }u\in \mathcal{U},
\end{equation*}
and
\begin{equation*}
    J''(\bar u)(u-\bar u)\geq -k\|u-\bar u\|_{L^1(\Omega)}^{1+\frac{1}{\gamma}} \text{ for all } (u-\bar u)\in C^\tau_{\bar u} \text{ with } \|u-\bar u\|_{L^1(\Omega)}<\alpha.
\end{equation*}
Then Assumption \ref{asuff2}, $\beta\in \{1/2,1\}$, holds for some appropriate constants. 
Further, let $\frac{\partial L_b}{\partial y}=0$. Then condition \eqref{sosc} implies Assumption \ref{asuffst}.
\end{Theorem}
\begin{proof}
   It is sufficient to prove the statement for the Assumption \ref{asuff2} on the cone $D^\tau_{\bar u}$. Thus, we only need to consider the case $(u-\bar u)\notin G^\tau_{\bar u}$.
But by definition of $(u-\bar u)\notin G^\tau_{\bar u}$, $J'(\bar u)(u-\bar u)>\tau \|z_{\bar u,u-\bar u}\|_{L^1(\Omega)}$. Using Theorem \ref{T3.1}, it is straight forward to  estimate for some constant $d$ independent of $u$
    \begin{equation*}
        \Big \vert J''(\bar u)(u-\bar u)^2\Big\vert \leq d\|z_{\bar u,u-\bar u}\|_{L^\infty(\Omega)}\|z_{\bar u,u-\bar u}\|_{L^1(\Omega)}.
    \end{equation*}
    By the assumption of this theorem, it also holds
    \begin{equation*}
    J'(\bar u)(u-\bar u)\geq c\|u-\bar u\|^{1+\frac{1}{\gamma}}_{L^1(\Omega)}.
\end{equation*}
Thus combining the estimates we obtain for $\|u-\bar u\|_{L^1(\Omega)}$ sufficiently small
\begin{align*}
&J'(\bar u)(u-\bar u)+J''(\bar u)(u-\bar u)^2\geq\frac{1}{2}J'(\bar u)(u-\bar u)+(\frac{1}{2}\tau -d\|z_{\bar u,u-\bar u}\|_{L^\infty(\Omega)})\|z_{\bar u,u-\bar u}\|_{L^1(\Omega)}\\
    &\geq c/2\|u-\bar u\|^{1+\frac{1}{\gamma}}_{L^1(\Omega)}+(\frac{1}{2}\tau -d\|z_{\bar u,u-\bar u}\|_{L^\infty(\Omega)})\|z_{\bar u,u-\bar u}\|_{L^1(\Omega)}\geq c/2\|u-\bar u\|^{1+\frac{1}{\gamma}}_{L^1(\Omega)}.
\end{align*}
The claim regarding Assumption \ref{asuffst} is straightforwardly obtained by similar arguments.
\end{proof}

\section{Discrete model and error estimates}\label{discmee}
We come to the main part of this manuscript. The goal is to prove error estimates for the numerical approximation under Assumption \ref{asuff2} for $\gamma\in (2/(n+2),1]$ and Assumptions \ref{asuffstm} and \ref{asuffst}.
\subsection{The finite element scheme} The finite element scheme we consider is close to the one in \cite{CM2021}; we also refer to \cite{BS2002} for an overview of the finite element method.
In this section, we assume $\Omega$ to be convex, see \cite[Section 5.2]{RT1983}. Let $\{\tau_h\}_{h>0}$ be a quasi-uniform family of triangulations of $\bar \Omega $. That is, for each $T\in \tau_h$, $\rho(T)$ denotes the diameter of $T$, and $\sigma(T)$ denotes the diameter of the largest ball inscribed in $T$. The mesh size is defined by $h:=\max_{T\in \tau_h}\rho(T)$. We assume that there exist two positive constants $\tilde \sigma$ and $ \tilde \rho$ such that
\begin{equation}
    \frac{\rho(T)}{\sigma(T)}\leq \tilde \sigma \text{ and } \frac{h}{\rho(T)} \leq \tilde \rho,
\end{equation}
for all $T\in \tau_h$ and all $h>0$.
Denote $\bar \Omega_h=\cup_{T\in\tau_h}T$ and {define $\Omega_h:=\text{int} \bar \Omega_h$},
and assume that every boundary node of $\Omega_h$ is a point of $\Gamma$. Suppose that there exists a constant $C_\Gamma>0$ independent of $h$ such that the distance $d_{\Gamma}$ satisfies $d_{\Gamma}(x)<C_{\Gamma}h^2$ for every $x\in \Gamma_h=\partial \Omega_h$. As a consequence, we infer the existence of a constant $C_\Omega>0$ independent of $h$ such that 
\begin{equation}\label{distset}
\vert \Omega\setminus \Omega_h\vert \leq C_{\Omega}h^2,
\end{equation}
where $\vert \cdot\vert $ denotes the Lebesgue measure, see \cite[(5.2.19)]{RT1983}. We define the finite-dimensional space 
\begin{equation*}
Y_h=\{z_h\in C(\bar \Omega): z_{h\vert T}\in P_1(T)\ \forall T\in \tau_h \ \textrm{ and } z_h\equiv 0\text{ on }\Omega \setminus \Omega_h \},
\end{equation*}
where $P_1(T)$ denotes the polynomials in $T$ of degree at most $1$.\\
For  $u\in L^2(\Omega)$, the associated discrete state is the unique element $y_h(u)\in Y_h$ that solves 
\begin{equation}
\label{disceq}
a(y_h,z_h)+\int_{\Omega_h}f(x,y_h)z_h\dx=\int_{\Omega_h}u z_h\dx \ \ \forall z_h\in Y_h,
\end{equation}
where 
\begin{equation*}
a(y,z)=\sum_{i,j=1}^n\int_\Omega a_{ij}\partial_{x_i}y\partial_{x_j}z\dx \ \ \forall y,z \in H^1(\Omega).
\end{equation*}
The proof of the existence and uniqueness of a solution for \eqref{disceq} is standard; see, for instance, \cite{CM2002}.
\begin{Lemma}\cite[Lemma 3]{CM2021}.
\label{discest}There exists a constant $c>0$, depending on the data of the problem but independent of the discretization parameter $h$, s. t. for every $u\in \mathcal U$
\begin{equation}\label{estforstate2}
\| y_h(u)-y_u\|_{L^2(\Omega)}\leq ch^2,
\end{equation}
\begin{equation}\label{estforstatelog}
\| y_h(u)-y_u\|_{L^\infty(\Omega)}\leq ch^2\vert \log h\vert^2.
\end{equation}
\end{Lemma}
The set of feasible controls for the discrete problem is given by
\begin{equation*}
U_h:=\{ u_h\in L^\infty(\Omega_h):u_{h\vert T}\in P_0(T)\ \forall T\in \tau_h\}.
\end{equation*}
By $\Pi_h$ we denote the linear projection onto  $U_h$ in the $L^2(\Omega_h)$ given by
\begin{equation*}
(\Pi_h u)_{\vert T}=\frac{1}{\vert T\vert}\int_T u\dx,\ \ \forall T\in \tau_h.
\end{equation*}
By $u_h\rightharpoonup u$ weak* in $L^\infty(\Omega)$ we mean, as in \cite{CM2021}, that
\begin{equation*}
    \int_{\Omega_h}u_hv\dx\to\int_\Omega uv\dx \ \ \forall \ v\in L^1(\Omega).
\end{equation*}
\begin{Lemma}{\cite[Lemma 4]{CM2021}}\label{dualsob}
Given $1 < p <\infty$ there exists a positive constant $\hat C_p$ that depends on $p$ and $\Omega$ but is independent of $h$ such that
\begin{equation*}
    \|u-\Pi_h u\|_{W^{-1,p}(\Omega_h)}\leq \hat C_p h\| u\|_{L^p(\Omega)}\ \ \forall \ u \in L^p(\Omega).
\end{equation*}
\end{Lemma}
We define $J_h(u):=\int_{\Omega_h} L(x,y_h(u),u)\dx$ and $\mathcal U_h:=U_h\cap \mathcal U$. Then the discrete problem is given by
\begin{equation}\label{discprob}
\min_{u_h\in \mathcal U_h}J_h(u_h).
\end{equation}
The set $\mathcal U_h$ is compact and nonempty, and the existence of a 
global solution of \eqref{discprob} follows from standard arguments.
For $u\in L^2(\Omega)$, the  discrete adjoint state $\varphi_h(u)\in Y_h$ is the unique solution of
\begin{equation}\label{discstate}
a(z_h,\varphi_h)+\int_{\Omega_h}\frac{\partial f}{\partial y}(x,y_h(u))\varphi_h z_h\dx=\int_{\Omega_h}\frac{\partial L}{\partial y}(x,y_h(u),u)z_h\dx\ \ \forall z_h\in Y_h.
\end{equation}
Again the proof of the existence and uniqueness of a solution for \eqref{disceq} is standard, see \cite{CM2002}.
One can calculate that $J'_h(u)(v)=\int_{\Omega_h}(L_b(x,y_h(u))+\varphi_h(u))v\dx$.
A local solution of \eqref{discprob} satisfies the variational inequality
\begin{equation*}
J'_h(\bar u_h)(u_h-\bar u_h )\geq 0\ \ \forall u_h\in \mathcal U_h.
\end{equation*}
In the following, similar as in \cite{CM2021}, we identify $u_h=\bar u$ on $\Omega\setminus \Omega_h$.
The existence of a sequence of solutions to the discrete problem that converges to an optimal solution of \eqref{ocp1} is provided in the next theorem.
\begin{Theorem}{\cite[Theorem 6]{CM2021}}
Let $\bar u$ be a strict strong local minimizer of \eqref{ocp1}. Then, there exists a sequence $\{\bar u_h\}_h$ of local minimizes of \eqref{discprob} such that $\bar u_h\rightharpoonup \bar u$ weak* in $L^\infty(\Omega)$.
Moreover, there exists $h_0> 0$ such that
\begin{equation}\label{equation:discconv}
    J_h(\bar u_h) \leq J_h(u_h) \quad \text{for all } u_h \in  \mathcal{U}_h \text{ with } \|y_h(u_h)- y_h(\bar u_h)\|_{L^\infty(\Omega_h)} \leq \rho, \text{ for all } h \leq h_0.
\end{equation}
Conversely, let $\{ \bar u_h\}_h$ be a sequence of local minimizers of \eqref{discprob} satisfying \eqref{equation:discconv} for some
given $\rho >0$ and such that $\bar u_h \rightharpoonup^* \bar u $ in $L^\infty(\Omega)$. Then $\bar u$ is a strong local solution of \eqref{ocp1}
satisfying
\begin{equation}
J (\bar u) \leq J (u) \text{ for all }u \in \mathcal{U} \text{ with } \|y_u -\bar y\|_{L^\infty(\Omega)} < \rho.
\end{equation}
\end{Theorem}
\begin{Remark}
    Let $\bar u\in L^r(\Omega)$ and $u_h\rightharpoonup^* \bar u$ in $L^\infty(\Omega)$. Then $\|y_h(\bar u_h)-y_{\bar u}\|_{L^\infty(\Omega)}\to 0$ as $h\to 0$. This follows trivially since the right-hand side of
    \begin{equation}
        \|y_h(\bar u_h)-y_{\bar u}\|_{L^\infty(\Omega)}\leq \|y_h(\bar u_h)-y_{\bar u_h}\|_{L^\infty(\Omega)}+\|y_{\bar u_h}-y_{\bar u}\|_{L^\infty(\Omega)}
    \end{equation}
    tends to zero for $h\to 0$ due to Theorem \ref{contandregsem} and Lemma \ref{discest}, see also the related  statement in the proof of \cite[Theorem 6]{CM2021}.
\end{Remark}

For the estimations for the variational discretization, we need the following theorem. The proof of Theorem \ref{interadj} is done along the proof of \cite[Theorem 9]{CM2021}\label{varthmest} using the arguments from the proof of \cite[Lemma 3]{CM2021}.
\begin{Theorem}\label{interadj}
Let $\bar u_h $ denote a solution to \eqref{discprob}. We denote by $y_{\bar u_h}$ and $ \varphi_{\bar u_h}$ the solution to the continuous state equation and to the corresponding adjoint equation with respect to $\bar u_h$. By $\varphi_h(\bar u_h)$ we denote the discrete adjoint equation corresponding to $\bar u_h$ and $\varphi^h_{\bar u_h}$ denotes the solution to the following equation
\begin{equation*}
		\left\{ \begin{array}{lll}
		\mathcal A^\ast \varphi+\frac{\partial f}{\partial y}(\cdot, y_h(\bar u_h))\varphi& =\frac{\partial L}{\partial y}(\cdot ,y_h(\bar u_h))\  &\text{ in }\ \Omega,\\
		\varphi&=0\ &\text{ on } \Gamma.\\
	\end{array} \right.
	\end{equation*}
	There exists a positive constant $c$, which depends on the data of the problem but is
independent of the parameter $h$ and a positive number $h_0$ such that for all $h<h_0$ it holds
	\begin{align}
	&\| \varphi_{\bar u_h}-\varphi^h_{\bar u_h}\|_{L^\infty(\Omega)}\leq ch^2,\label{distadj1}\\
	&\| \varphi_h(\bar u_h)-\varphi^h_{\bar u_h}\|_{L^\infty(\Omega)}\leq ch^2 \vert \log h\vert^2.
	\label{distadj2}
	\end{align}
\end{Theorem}
\subsection{Discretization with piece-wise constant controls}\label{point}
The two main goals of this section are to prove that Assumptions \ref{asuff2}, \ref{asuffstm} and \ref{asuffst} allow finite element error estimates and to improve the error estimates in the literature for $\gamma\in (0,1)$. Before stating the main theorems, let us consider two preliminary lemmas.
Let us recall that $\bar H_u(x):=H_u(x, \bar y (x),  \bar \varphi (x), \bar u(x))=\bar \varphi(x)+L_b(x,\bar y(x))$.
The assumption that  $\bar H_u$ is Lipschitz is not a significant constraint for tracking-type objective functionals where $\frac{\partial L_b}{\partial y}=0$. This is because of the assumptions on the control problem in this section; the adjoint state, $\varphi_{\bar u}$, is already Lipschitz continuous. If $ L_b\ne 0$, the significance of the constraints depends on the regularity of $L_b(\cdot,y)$, which comes down to the regularity of the state $y$.
Let us also recall that $\bar \varphi\in W^{2,p}(\Omega)$ as a consequence of the regularity of the domain, the regularity of the coefficients of the elliptic operator, and the boundedness of the right-hand side of the adjoint equation due to the boundedness of the solution to the state equation. Furthermore, the solution to the state equation, $\bar y$ is due to Theorem \ref{contandregsem} in $W^{2,p}(\Omega)$ and thus Lipschitz as well. Finally, due to Assumption \ref{A3}, which gives us the needed regularity of $L_b$, we can infer the Lipschitz continuity of $\bar H_u$.

\begin{Lemma}\label{lipest} Consider a bang-bang control $\bar u\in \mathcal U$ satisfying the first order optimality condition.\\
Then
$ J'(\bar u)(\Pi_h\bar u-\bar u)\leq \mathrm{Lip}_{\bar H_u}h\|\bar u-\Pi_h\bar u\|_{L^1(\Omega_h)}$.
\end{Lemma}
\begin{proof}
    For the reader's convenience, we present a proof which follows the arguments in \cite[Lemma 7]{CM2021}.
    Let $T$ be a {triangle} such that  $\bar H_u$ changes its sign {in $T$}. Since $\bar H_u$ is Lipschitz continuous, there exists a point $x_0 \in T$ with $\bar H_u (x_0)=0$. For $x\in T$ we obtain
    \begin{equation}\label{sigmaest}
    \vert \bar H_u(x) \vert =     \vert \bar H_u(x)-\bar H_u(x_0) \vert\leq Lip_{{\bar H_u}} \vert x-x_0\vert \leq Lip_{\bar H_u} h.
    \end{equation}
    Let us denote by $S_h$ the union of elements $T$ such that $\bar H_u$ changes the sign. Then on the set $S_h$, we have the estimate $\|\bar H_u\|_{L^\infty(S_h)}\leq Lip_{\bar H_u}h$. If $\bar H_u$ does not change the sign on an element $T$, the bang-bang structure implies $\Pi_h\bar u=\bar u $ on $\Omega\setminus S_h$. As a consequence, we obtain the estimate
    \begin{align*}
        J'(\bar u)(\Pi_h\bar u-\bar u)=\int_{\Omega_h} {\bar H_u}(\Pi_h\bar u-\bar u)\ \text{ d}x\leq Lip_{\bar H_u}h \|\Pi_h\bar u-\bar u\|_{L^1(S_h)}.
    \end{align*}
\end{proof}
The next lemma estimates the $L^1$-distance of a bang-bang reference solution and its projection. It is needed to obtain error estimates for the numerical approximation later on.
\begin{Lemma}\label{difdccc}
Let $\bar u\in \mathcal U$ satisfy Assumption \ref{asuff2}$(\beta=1/2)$, here we allow $\gamma \in (0,1]$. There exists positive constants $c$ (independent of $h$) and a $h_0$, such that for all $h<h_0$:
\begin{equation}\label{distcont}
\| \bar u-\Pi_h \bar u\|_{L^1(\Omega_h)}\leq c h^\gamma.
\end{equation}
\end{Lemma}
\begin{proof}
Since $\bar u\in \mathcal U$ satisfies Assumption \ref{asuff2}, there exist positive constants $\kappa$ and  $\alpha$ such that 
\begin{equation*}
J'(\bar u)(\Pi_h \bar u-\bar u)+1/2 J''(\bar u)(\Pi_h \bar u-\bar u)^2\geq \kappa\| \Pi_h \bar u-\bar u\|_{L^1(\Omega_h)}^{1+1/\gamma}\ \ \forall h \textrm{ s.t. } \|\Pi_h \bar u-\bar u\|_{L^1(\Omega_h)}<\alpha.
\end{equation*}
By Lemma \ref{lipest}, we obtain
\begin{equation*}
 J'(\bar u)(\Pi_h \bar u-\bar u)=\int_\Omega (\varphi_{\bar u}+L_b(\cdot,y_{\bar u})) (\Pi_h \bar u-\bar u)\dx\leq \textrm{Lip}_{\textcolor{red}{\bar H_u}}h\|\Pi_h \bar u-\bar u\|_{L^1(\Omega_h)}.
\end{equation*}
We recall that according to Theorem \ref{T3.1}, it holds
\begin{equation}
    \begin{aligned}
        J''(\bar u)(\Pi_h \bar u-\bar u)^2& = \int_\Omega\Big[\frac{\partial^2L}{\partial y^2}(\cdot,y_{\bar u},\bar u) - \varphi_{\bar u}\frac{\partial^2f}{\partial y^2}(\cdot ,y_{\bar u})\Big]
       z_{\bar u,\Pi_h \bar u-\bar u}^2\dx \\
       &+ 2\int_\Omega \Big[{\frac{\partial L_b}{\partial y}(\cdot,y_{\bar u})}\Big]z_{\bar u,\Pi_h \bar u-\bar u}(\Pi_h \bar u-\bar u)\dx=I_1+I_2.
    \end{aligned}
\end{equation}
Given $n<p$, $n/2<r$, using {Assumption \ref{A3}, Lemma \ref{L2.2}, Lemma \ref{L2.3},  Lemma \ref{stsupdual} and Lemma \ref{dualsob}}, the first term of the second variation is estimated by
\begin{align*}
\vert I_1 \vert&\leq
 \Big\|\frac{\partial^2 L}{\partial ^2 y}(\cdot,y_{\bar u})-p_{\bar u}\frac{\partial^2 f}{\partial^2 y}(\cdot, y_{\bar u})\Big\|_{L^\infty(\Omega)}\|z_{\bar u,\Pi_h \bar u-\bar u}\|_{L^2(\Omega)}^2\\
 &\leq 2(C_{L,M}+C_{L,M}C_{f,M}\vert \Omega\vert^{1/r})\vert \Omega \vert^{1/2} \|y_{\Pi\bar u}- y_{\bar u}\|_{L^\infty(\Omega)}\|z_{\bar u,\Pi_h \bar u-\bar u}\|_{L^2(\Omega)}\\
 &\leq 2C_2 \check C_p(C_{L,M}+C_{L,M}C_{f,M}\vert \Omega\vert^{1/r})\vert \Omega \vert^{1/2} \|\Pi_h \bar u-\bar u\|_{W^{-1,p}(\Omega_h)} \|\Pi_h \bar u-\bar u\|_{L^1(\Omega_h)}\\
 &\leq 2C_2 \check C_p \hat C_p(C_{L,M}+C_{L,M}C_{f,M}\vert \Omega\vert^{1/r})\vert \Omega \vert^{(2+r)/(2r)} \max\{\vert u_a\vert,\ \vert u_b\vert \}h \|\Pi_h \bar u-\bar u\|_{L^1(\Omega_h)}.
\end{align*}
It is left to estimate the term in the second line. Again using Assumption \ref{A3}, Lemma \ref{L2.3} and Lemma Lemma \ref{dualsob}, we obtain
\begin{equation}
    \begin{aligned}
       \vert I_2 \vert&\leq 2C_{L,M} \| z_{\bar u,\Pi_h \bar u-\bar u}\|_{L^\infty(\Omega)} \|\Pi_h \bar u-\bar u \|_{L^1(\Omega_h)}\\
        &\leq 4 C_{L,M} \| y_{\Pi\bar u} -y_{\bar u}\|_{L^\infty(\Omega)} \|\Pi_h \bar u-\bar u \|_{L^1(\Omega_h)}\\
        &\leq 4 C_{L,M}\hat C_p \vert \Omega\vert^{1/r} \max\{\vert u_a\vert,\ \vert u_b\vert \}h \|\Pi_h \bar u-\bar u \|_{L^1(\Omega_h)}.
    \end{aligned}
\end{equation}
Thus, we infer the existence of a positive constant $c$ such that 
\begin{equation*}
    \kappa\| \Pi_h \bar u-\bar u\|_{L^1(\Omega_h)}^{1+1/\gamma} \leq c h \|\Pi_h \bar u-\bar u \|_{L^1(\Omega_h)}.
\end{equation*}
Dividing both sides by $\|\Pi_h \bar u-\bar u \|_{L^1(\Omega_h)}$, completes the proof.
\end{proof}

Now, we are ready to state the main theorems of this section.
\begin{Theorem}\label{pointest}
Let $\bar u$ be a local solution of \eqref{const}-\eqref{see1}. Consider the constant $\alpha$ corresponding to the Assumptions \ref{asuff2}, \ref{asuffstm} or \ref{asuffst}. Consider a sequence of discrete optimal controls $\bar u_h\in \mathcal U_h$ of \eqref{discprob} that satisfy  $\|\bar u_h-\bar u\|_{L^1(\Omega_h)}< \alpha$. 
\begin{enumerate}
\item  Let $L_b=0$ in the objective functional and let $\bar u$ satisfy Assumption \ref{asuffst}$(\beta=1/2)$.
Then, there exists a positive constant $c$ independent of $h$ and a $h_0$ such that 
\begin{equation}\label{erroot}
\|{y_h(\bar u_h)}-\bar y \|_{L^2(\Omega)}+{\|\varphi_h(\bar u_h)-\bar \varphi \|_{L^2(\Omega)}}\leq c\sqrt{h}   \ \text{  for all } h\leq h_0.
\end{equation}
\item  Let {$\frac{\partial L_b}{\partial y}=0$} in the objective functional and let $\bar u$ satisfy Assumption \ref{asuffstm}$(\beta=1/2)$.
Then, there exists a positive constant $c$ independent of $h$ and a $h_0$ such that 
\begin{equation}\label{errootm}
\|{y_h(\bar u_h)}-\bar y \|_{L^2(\Omega)}+{\|\varphi_h(\bar u_h)-\bar \varphi \|_{L^2(\Omega)}}\leq {c\sqrt{h^2+h\| \Pi_h \bar u-\bar u\|_{L^1(\Omega_h)}}}   \ \text{  for all } h\leq h_0.
\end{equation}
\item Let $\bar u$ satisfy Assumption \ref{asuff2}$(\beta=1/2)$.
Then, there exists a positive constant $c$ independent of $h$ and a $h_0$ such that 
\begin{equation}\label{erlipc}
\|\bar u_h-\bar u\|_{L^1(\Omega_h)}+\|{y_h(\bar u_h)}-\bar y\|_{L^2(\Omega)}+{\|\varphi_h(\bar u_h)-\bar \varphi \|_{L^2(\Omega)}}\leq ch^\gamma   \ \text{  for all } h\leq h_0.
\end{equation}
If additionally Assumption \ref{asuffst}$(\beta=1/2)$ holds then \eqref{erlipc} holds for $\gamma\in (0,1]$.
\end{enumerate}
\end{Theorem}
\begin{proof}
Let us consider a discrete control $\bar u_h$ that satisfies the theorem's assumptions. We first prove the existence of a positive constant $c$ such that
\begin{align*}
     J(\bar u_h)-J(\bar u)&\leq \big\vert J(\bar u_h)-J_h(\bar u_h)\big\vert + J_h(\bar u_h)-J_h(\Pi_h\bar u)
     +\big\vert J_h(\Pi_h\bar u)-J(\Pi_h\bar u)\big\vert\\
     &+\big\vert J(\Pi_h\bar u)-J(\bar u)\big\vert =\vert I_1\vert +{ I_2}+\vert I_3\vert +\vert I_4\vert \leq ch^{1+\gamma}.
\end{align*}
To estimate the first term, we use Assumption \ref{A3}, the estimates in Lemma \ref{discest}, \eqref{distset} and the mean value theorem to obtain for intermediate functions $y_\theta$ and $y_\vartheta$ that
\begin{align*}
\vert I_1\vert &=\Bigg\vert\int_{\Omega\setminus \Omega_h} L(x,y_{ \bar u_h}, \bar u_h)\dx+\int_{\Omega_h} L(x,y( \bar u_h), \bar u_h)-L(x,y_{ \bar u_h}, \bar u_h)\dx\Bigg\vert\\
    &\leq \Big(\Big \| \frac{\partial L_a}{ \partial y}(\cdot,y_\theta)\Big\|_{L^2(\Omega)}+\Big\|\frac{\partial L_b}{\partial y}(\cdot,y_\vartheta)\bar u_h\Big\|_{L^2(\Omega_h)} +C_\Omega \|L(\cdot,y_{ \bar u_h}, \bar u_h)\|_{L^\infty(\Omega)}\Big)h^2\\
    &\leq (C_{L,M}+C_{L,M}\max\{\vert u_a\vert,\ \vert u_b\vert \}+C_\Omega K_{u_a,u_b, T_r} )h^2,
\end{align*}
where the constant $K_{u_a,u_b, T_r}$ with $\|L(\cdot,y_{u}, u)\|_{L^\infty(\Omega)}\leq K_{u_a,u_b, T_r} $ for all $u\in \mathcal{U}$, as indicated by the subscripts, depends on the control model though $u_a$,$u_b$ and  $T_r$. We have {$I_2\leq 0$} for the second term since $\bar u _h$ is a minimizer of \eqref{discprob}.
The term $I_3$ can be estimated similarly as the first term, 
\begin{align*}
    \vert I_3\vert &=\Bigg\vert \int_{\Omega} L(x,y_{\Pi_h \bar u},\Pi_h \bar u)\dx-\int_{\Omega_h} L(x,y(\Pi_h \bar u),\Pi_h \bar u)\dx\Bigg\vert \\
    &=\Bigg\vert \int_{\Omega\setminus \Omega_h} L(x,y_{\Pi_h \bar u},\Pi_h \bar u)\dx+\int_{\Omega_h} L(x,y_{\Pi_h \bar u},\Pi_h \bar u)-L(x,y(\Pi_h \bar u),\Pi_h \bar u)\dx\Bigg\vert\\
    &\leq\Big(\Big \| \frac{\partial L_a}{ \partial y}(x,y_\theta)\Big\|_{L^2(\Omega)}+\Big\|\frac{\partial L_b}{\partial y}(x,y_\vartheta)\Pi_h\bar u\Big\|_{L^2(\Omega_h)} +C_\Omega \|L(x,y_\vartheta,\Pi_h\bar u)\|_{L^\infty(\Omega)}\Big) h^2\\
    &\leq (C_{L,M}+C_{L,M}\max\{\vert u_a\vert,\ \vert u_b\vert \}+C_\Omega K_{u_a,u_b, T_r} )h^2.
\end{align*}
We come to the crucial part of the proof, the estimate of the last term $I_4$. The estimation of the term $I_4$ determines the overall convergence rate since the other terms already satisfy the good rate $\vert I_1\vert, \vert I_2\vert \leq c h^2$. To shorten the notation, let us denote by $L_{a,y}$ and $L_{b,y}$ the derivatives of $L_a$ and $L_b$ by $y$. By the mean value theorem, for some intermediate function $y_\theta$ we have
\begin{align*}
 &I_4= \int_\Omega L(x,y_{\Pi_h(\bar u)},\Pi_h(\bar u))-L(x,y_{\bar u},\bar u)\dx=\int_\Omega L_{a,y}(x,y_\theta)(y_{\Pi_h\bar u}-y_{\bar u})\dx\\
 &+\int_\Omega L_{b,y}(x,y_\theta)(y_{\Pi_h\bar u}-y_{\bar u})\Pi_h \bar u\dx
+\int_\Omega L_{b}(x,y_{\bar u})(\Pi_h \bar u-\bar u)\dx\\
&=\int_\Omega L_{a,y}(x,y_{\bar u})(y_{\Pi_h\bar u}-y_{\bar u})\dx+\int_\Omega L_{b,y}(x,y_{\bar u})(y_{\Pi_h\bar u}-y_{\bar u}) \bar u\dx\\
&+\int_\Omega L_{b}(x,y_{\bar u})(\Pi_h \bar u-\bar u)\dx+\int_\Omega (L_{a,y}(x,y_\theta)-L_{a,y}(x,y_{\bar u}))(y_{\Pi_h\bar u}-y_{\bar u})\dx \\
&+\int_\Omega L_{b,y}(x,y_{\bar u})(y_{\Pi_h\bar u}-y_{\bar u})(\Pi_h \bar u-\bar u)\dx+\int_\Omega ( L_{b,y}(x,y_\theta)-L_{b,y}(x,y_{\bar u}))(y_{\Pi_h\bar u}-y_{\bar u})\Pi_h \bar u\dx.
\end{align*}
Thus,
\begin{align*}
&I_4=\Big[\int_\Omega (L_{a,y}(x,y_{\bar u})+L_{b,y}(x,y_{\bar u})\bar u)z_{\bar u,\Pi_h \bar u-\bar u}\dx\Big]
+\Big[ \int_\Omega L_{b}(x,y_{\bar u})(\Pi_h \bar u-\bar u)\dx\Big]\\
&+\Big[\int_\Omega (L_{a,y}(x,y_{\bar u})+L_{b,y}(x,y_{\bar u})\Pi_h \bar u)(y_{\Pi_h\bar u}-y_{\bar u}-z_{\bar u,\Pi_h \bar u-\bar u})\dx\Big]\\
&+\Big[\int_\Omega (L_{a,y}(x,y_\theta)-L_{a,y}(x,y_{\bar u}))(y_{\Pi_h\bar u}-y_{\bar u})\dx\Big]+\Big[\int_\Omega ( L_{b,y}(x,y_\theta)-L_{b,y}(x,y_{\bar u}))(y_{\Pi_h\bar u}-y_{\bar u})\Pi_h \bar u\dx\Big]\\
&+\Big[\int_\Omega L_{b,y}(x,y_{\bar u})(y_{\Pi_h\bar u}-y_{\bar u})(\Pi_h \bar u-\bar u)\dx\Big]=\sum_{i=1}^6 K_i.
\end{align*} Let us first consider the arguments for the estimation \eqref{erlipc}.
We remind that since $\bar u$ satisfies Assumption \ref{asuff2}, it is bang-bang. We estimate the terms $K_1$ and $K_2$ together, that is, integrating by parts, using Lemma \ref{lipest} and Lemma \ref{difdccc} guarantee the existence of a positive constant $c$ such that
\begin{align*}
\vert K_1+K_2\vert&={\Big\vert\int_\Omega (p_{\bar u}+ L_b(x,y_{\bar u}))(\Pi_h\bar u-\bar u)\dx\Big\vert}\\
&\leq \textrm{Lip}_{\bar H_u}\|\bar H_u \|_{L^\infty(\Omega)} \|\Pi_h \bar u-\bar u\|_{L^1(\Omega_h)}\leq c\textrm{Lip}_{\bar H_u}h^{1+\gamma}.
\end{align*}
The term $K_3$ is estimated, using Assumption \ref{A3}, Lemma \ref{L2.3},  Lemma \ref{stsupdual} and Lemma \ref{dualsob},
\begin{align*}
\vert K_3\vert &\leq C_{L,M}(1+ \max\{\vert u_a\vert, \vert u_b\vert \}) \|y_{\Pi_h\bar u}-y_{\bar u}-z_{\bar u,\Pi_h \bar u-\bar u}\|_{L^1(\Omega)}^2\\
&\leq C_{L,M}M_1(1+ \max\{\vert u_a\vert, \vert u_b\vert \})   \|y_{\Pi_h\bar u}-y_{\bar u}\|_{L^2(\Omega)}^2\\
&\leq C_{L,M}\check C_p^2 M_1(1+ \max\{\vert u_a\vert, \vert u_b\vert \}) \vert \Omega \vert \| \Pi_h \bar u-\bar u\|_{W^{-1,p}(\Omega_h)}^2\\
&\leq C_{L,M}\check C_p^2 \hat C_p^2 M_1(1+ \max\{\vert u_a\vert, \vert u_b\vert \}) \vert \Omega \vert^{1+2/p} \max\{\vert u_a\vert, \vert u_b\vert \}^2h^2.
\end{align*}
For the estimation of the term $K_4$ we use that due to Assumption \ref{A3}, $L_{y, a}$ is locally Lipschitz continuous, with Lipschitz constant denoted by $\textrm{Lip}_{L_{a,y};M}$. We obtain using again Lemma \ref{L2.3},  Lemma \ref{stsupdual} and Lemma \ref{dualsob} that 
\begin{align*}
    \vert K_4\vert &\leq \textrm{Lip}_{L_{y,a};M} \vert \Omega \vert   \|y_{\Pi_h\bar u}-y_{\bar u}\|_{L^\infty(\Omega)}^2\\
    &\leq \textrm{Lip}_{L_{a,y};M} \check C_p^2 \hat C_p^2 \vert \Omega \vert^{1+2/p}  \max\{\vert u_a\vert, \vert u_b\vert \}^2  h^2.
\end{align*}

Denoting the local Lpischitz constant of $L_{b,y}$ by $ \textrm{Lip}_{L_{b,y};M}$, the term $K_5$ is estimated using the same arguments by
\begin{align*}
    \vert K_5 \vert & \leq \textrm{Lip}_{L_{b,y};M} \check C_p^2 \hat C_p^2 \vert \Omega \vert^{1+2/p}  \max\{\vert u_a\vert, \vert u_b\vert \}^3  h^2.
\end{align*}
Finally, for the term $I_6$,using Assumption \ref{A3}, Lemma \ref{L2.3},  Lemma \ref{stsupdual} and Lemma \ref{dualsob}, we estimate
\begin{align*}
\vert K_6\vert &\leq C_{L,M}\|y_{\Pi_h\bar u}-y_{\bar u}\|_{L^\infty(\Omega)}\| \Pi_h \bar u-\bar u\|_{L^1(\Omega_h)}\\
&\leq C_{L,M} \check C_p\| \Pi_h \bar u-\bar u\|_{W^{-1,p}(\Omega_h)} \| \Pi_h \bar u-\bar u\|_{L^1(\Omega_h)}\\
&\leq C_{L,M} \check C_p \hat C_p \max\{\vert u_a\vert, \vert u_b\vert \}\vert \Omega \vert^{1/p} c  h^{1+\gamma}.
\end{align*}
To complete the proof of \eqref{erlipc}, we conclude from the estimates of the terms $I_i$, $i\in \{1,...,6\}$, Theorem \ref{eqhalf} and Theorem \ref{coer1}, that there exist positive constants $\kappa$,$k$ and $\alpha$ such that 
\begin{equation*}
     \kappa\|\bar u_h-\bar u\|_{L^1(\Omega_h)}^{1+\frac{1}{\gamma}} \leq J(\bar u_h)-J(\bar u) \leq  kh^{1+\gamma} \ \text{ for all } \bar u_h \text{ with }\|\bar u_h-\bar u\|_{L^1(\Omega_h)}<\alpha.
\end{equation*}
This is equivalent to: $ (\kappa/k)^{\frac{\gamma}{\gamma+1}}h^\gamma\geq\|\bar u_h-\bar u\|_{L^1(\Omega_h)} \ \text{ for all } \bar u_h \text{ with }\|\bar u_h-\bar u\|_{L^1(\Omega_h)}<\alpha$.
To estimate the states and adjoint states, we argue as follows. For the states we see by Lemma \ref{L2.2} applied to $z_v:=\bar y-y_{\bar u_h}$, $v=\bar u-\bar u_h$, and Lemma \ref{discest}, \eqref{estforstate2}, that
\begin{equation*}
    \begin{aligned}
        \|\bar y- y_h(\bar u_h)\|_{L^2(\Omega)}&\leq \|\bar  y-y_{\bar u_h}\|_{L^2(\Omega)}+\| y_{\bar u_h}-y_h(\bar u_h)\|_{L^2(\Omega)}\\
        &\leq \hat C_2 \|\bar u-\bar u_h\|_{L^1(\Omega_h)}+\tilde C_2 h^2\leq \hat C_2 (\kappa/k)^{\gamma/(\gamma+1)}h^\gamma+\tilde C_2 h^2.
    \end{aligned}
\end{equation*}
For the estimate of the adjoint states, we use that as a consequence of Lemma \ref{L2.2} and $\|\bar  p-p_{\bar u_h}\|_{L^2(\Omega)}\leq C_2 \|\bar  y-y_{\bar u_h}\|_{L^2(\Omega)}$ that
\begin{equation*}
    \begin{aligned}
         \|\bar \varphi- \varphi_h(\bar u_h)\|_{L^2(\Omega)}&\leq \|\bar  \varphi-\varphi_{\bar u_h}\|_{L^2(\Omega)}+\| \varphi_{\bar u_h}-\varphi_h(\bar u_h)\|_{L^2(\Omega)}\\
        &\leq \hat C_2 \|\bar u-\bar u_h\|_{L^1(\Omega_h)}+\tilde C_2 h^2\leq \hat C_2 (\kappa/k)^{\gamma/(\gamma+1)}h^\gamma+\tilde C_2 h^2.
    \end{aligned}
\end{equation*}
and the proof of \eqref{erlipc} is complete.
Let us briefly comment on the procedure for the other claims \eqref{erroot} and \eqref{errootm}. Let us first consider \eqref{erroot}. Here, we do not assume that $\bar u $ is bang-bang. The terms $I_1$ and $I_3$ can be estimated by the same arguments as before thus we infer the existence of a positive constant $c$ such that 
\begin{equation*}
    \vert I_1\vert , \vert I_2 \vert \leq c h^2.
\end{equation*}
The estimation of the term $I_4$ is substantially easier due to the assumption $L_b=0$. We only need to consider the terms $K_1$, $K_3$ and $K_4$. For $K_1$ we estimate using Assumption \ref{A3}, Lemma \ref{L2.3}, Lemma \ref{stsupdual} and Lemma \ref{dualsob} to obtain
\begin{equation*}
    \begin{aligned}
        \vert K_1 \vert &\leq C_{L,M} \| z_{\bar u, \Pi_h \bar u-\bar u}\|_{L^2(\Omega)}\leq 3/2C_{L,M} \vert \Omega\vert^{1/2}\|y_{\Pi_h \bar u} -\bar y\|_{L^\infty(\Omega)}\\
        &\leq 2/3  C_{L,M} \check C_p \hat C_p\Omega\vert^{1/2}\|\Pi_h \bar u -\bar u\|_{W^{-1,p}(\Omega)}\leq 2/3  C_{L,M} \check C_p \hat C_p\Omega\vert^{1/2+1/p}\max\{\vert u_a\vert, \vert u_b\vert \} h.
    \end{aligned}
\end{equation*}
The terms $K_3$ and $K_4$ are estimated in the same way as before.
Thus, in total, we proved the existence of a positive constant c such that
\begin{equation*}
    \kappa\| z_{\bar u,\bar u_h-\bar u}\|_{L^2(\Omega)}^2\leq c h \ \text{ for all } \bar u_h \text{ with }\|y_{\bar u_h}-y_{\bar u}\|_{L^\infty(\Omega)}<\alpha.
\end{equation*}
This of course using Lemma \ref{L2.3} implies that $ \|\bar y- y_h(\bar u_h)\|_{L^2(\Omega)} \leq c h^{1/2}$. Now, the adjoint states can be estimated as argued above.
Finally, let us consider \eqref{errootm}. The terms $I_1$ and $I_3$ are estimated as before. Since $\frac{\partial L_b}{\partial y}=0$ we only have to estimate the terms $K_i$, $i\in \{1,...,4\}$ in $I_4$. Since $\bar u$ is bang-bang according to Proposition \ref{mixedbangbang}, we can employ Lemma \ref{lipest} to infer
\begin{align*}
\vert K_1+K_2\vert&=\Big\vert\int_\Omega (\varphi_{\bar u}+ L_b(x,y_{\bar u}))(\Pi_h\bar u-\bar u)\dx\Big\vert\\
&\leq c\textrm{Lip}_{\bar H_u} \|\Pi_h \bar u-\bar u\|_{L^1(\Omega_h)}h.
\end{align*}
The terms $K_3$ and $K_4$ are estimated as before. Thus, all in all, we have a positive constant $c$ such that
\begin{equation*}
   \frac{\kappa}{C_2}\|z_{\bar u,u-\bar u}\|_{L^2(\Omega)}^2\leq \kappa\|z_{\bar u,u-\bar u}\|_{L^2(\Omega)}\|\bar u-u\|_{L^1(\Omega)}\leq   J(u)-J(\bar u)\leq ch( h+\|\Pi_h\bar u-\bar u\|_{L^1(\Omega)}),
\end{equation*}
which, estimating the adjoint states as above, completes the proof.
\end{proof}

Under some mild additional assumption on the zero level set of $\bar H_u$ that exclude the appearance of singular arcs, we can significantly improve the result of Theorem \ref{pointest} using the next lemma instead of Lemma \ref{difdccc}. In what follows, we denote by $\mathcal H^m$, the $m$-dimensional Hausdorff measure, see \cite[Chapter 2]{EG1992}.

\begin{Lemma}\label{discretest}
Let $\bar u\in \mathcal{U}$ be bang-bang. Let us denote by A the points where {$\bar H_u=0$}.
Assume that $A$ consists of a finite union of $C^1$ curves $(n=2)$ or $C^1$-hypersurfaces $(n=3)$.
Then there exist  positive constants $c$ (independent of $h$) and $h_0$ such that for all $h\leq h_0$:
\begin{equation}
     \|\bar u-\Pi_h\bar u\|_{L^1(\Omega_h)}\leq c h. 
\end{equation}
\end{Lemma}

\begin{proof}
    First we notice that $\mathcal{H}^{n-1}(A)<\infty$.
 Let $\bar u$ be bang-bang and let, as before, denote by $\mathcal{S}_h$ the collection of the elements where there exists $x\in \text{int } T$ with $\bar H_u(x)=0$ and denote, as in Lemma \ref{lipest}, by $S_h$ the union of those elements. Then 
\begin{align*}
   & \|\bar u-\Pi_h\bar u\|_{L^1(\Omega_h)}= \|\bar u-\Pi_h\bar u\|_{L^1(S_h)}\\
    &\leq \sum_{T\in \mathcal S_h}\int_{T}\Big\vert \bar u(x)- \frac{1}{\mathcal{H}^n( T)}\int_{T} \bar u(y)\text{ d}y \Big\vert\text{d}x=\sum_{T\in \mathcal S_h}\int_T\frac{1}{\mathcal{H}^n( T)} \Big \vert \int_T\bar u(x)- \bar u(y)\text{ d}y \Big\vert \text{d}x\\
    &\leq \sum_{T\in \mathcal S_h}\int_{T}\frac{1}{\mathcal{H}^n( T)} \int_{T} \vert\bar u(x)- \bar u(y)\vert \text{ d}y \text{d}x\leq \|u_b-u_a\|_{L^\infty(\Omega
    )} \sum_{T\in \mathcal S_h} \mathcal{H}^n( T).
\end{align*}
 If the diameter $h$ of the quasi-uniform triangulations is sufficiently small, the number of elements that cover the set $A$ can be estimated by the quotient of the diameter $h$ and $\mathcal H^{n-1}(A)$. That is, there exists a constant $c$ (independent of $h$) and a $h_0$ such that for all $h\leq h_0$
\begin{equation*}
    \sum_{T\in \mathcal S_h}1\leq c \frac{ \mathcal H^{n-1}(A)}{h^{n-1}}.
\end{equation*}
This can be seen by the following arguments for the 2-dimensional case. 
The set $A=\cup_{i=1}^m A_i$ consists of a finite union of $C^1$-curves. Let us consider a given curve $A_i$ and a triangle $T\in \mathcal S_h$; we realize that $A_i$ intersects triangles $T$ such that for any given $x\in A_i$, $\max_{ y\in T} d(x, y)\leq h$. Now take for each $x\in A_i$ the unit normal $\nu(x)$ to $A_i$ and define the set $E_i^x:=\{x+\xi\nu(x)\vert \xi\in [-h,h]\}$ and $E_i:=\cup_{x\in A_i}E_i^x$. Then $\mathcal{H}^2(E_i)=2h\cdot {\mathcal H^1(A_i)}$. 

On the other hand, due to the quasi uniformity of the triangulation, there exists a positive constant $\bar c$ such that the measure of the triangles $T$ is uniformly bounded from below by $\bar c h^2$. Therefore there are at most 
\begin{equation*}
    \frac{\mathcal H^2(E_i)}{\bar ch^2}=\frac{2\mathcal H^1(A_i)h}{\bar ch^2}=\frac{2\mathcal H^1(A_i)}{\bar ch},
\end{equation*}
triangles that intersect the curve $A_i$. We obtain the claim by applying this to all the arcs $A_i$. The three-dimensional case follows by straightforward adaptions of the argument.
From here, using that for a positive constant $\tilde c$, the measure of the elements $T$ can be uniformly estimated by $\vert T \vert \leq \tilde c h^{n}$, we conclude:
\begin{align*}
   & \|\bar u-\Pi_h\bar u\|_{L^1(\Omega_h)}= \|\bar u-\Pi_h\bar u\|_{L^1(S_h)}\leq \|u_b-u_a\|_{L^\infty(\Omega
    )} \sum_{T_h\in \mathcal S_h} \mathcal{H}^n( T)\\
    &\leq \|u_b-u_a\|_{L^\infty(\Omega
    )}\frac{c \mathcal H^{n-1}(A)}{ h^{n-1}} \max_{T\in \mathcal S_h} \mathcal{H}^n( T)\leq \|u_b-u_a\|_{L^\infty(\Omega
    )}\frac{ \mathcal H^{n-1}(A)}{h^{n-1}} c\tilde c h^{n}.
\end{align*}
\end{proof}
    {Due to the assumption on the optimal control problem in this paper, the adjoint $\bar \varphi$ has regularity $W^{2,p}(\Omega)$, $p<\infty$. We remark that due to \cite[Section 2]{G1985}, for this result, it is necessary to have a $C^{1,1}$ boundary.} Then, due to the $W^{2,p}(\Omega)$ regularity of the adjoint, in dimension $n=2$, the Morse-Sard theorem for Sobolev functions and the implicit function theorem implies that for almost all $t$ in the image of $\bar \varphi$, the level set $[\bar \varphi=t]$ consists of finitely many disjoint $C^1$ simple curves \cite{BKK,D2001,F}. This almost everywhere result can be improved if $\bar \varphi$ satisfies
     \begin{equation}\label{gradgro}
    \min_{x\in A} \vert \nabla \bar \varphi(x)\vert  > 0,\ \ \ A = \{x\in \Omega \vert\ \bar \varphi(x) =0\}.
    \end{equation}
    Indeed, if \eqref{gradgro} is satisfied, $[\bar \varphi=0]$ consists of a finite union of simple $C^1$ curves, see \cite[Proposition 2.4, Corollary 2.12]{CNR}.
    This supports the assumption of Lemma \ref{discretest}. On the other hand, if $\bar \varphi\in C^1(\bar \Omega)$ satisfies \eqref{gradgro}, it already holds $\vert \{x\in \Omega\vert\ \vert \bar \varphi\vert \leq \varepsilon\}\vert \leq c\varepsilon$, 
    see \cite[Lemma 3.2]{DH2012b}. We apply Lemma \ref{discretest} when we only expect 
    \begin{equation*}
        \vert \{x\in \Omega\vert\ \vert \bar \varphi\vert \leq \varepsilon\}\vert \leq c\varepsilon^\gamma, \ \ \gamma\in (0,1),
    \end{equation*} which permits $\min_{x\in A} \vert \nabla \bar\varphi(x)\vert = 0$.
\bigskip

We obtain the following improvement of the estimation in Theorem \ref{pointest}.

\begin{Theorem}\label{pointest2}
Let $\bar u$ be a local solution of \eqref{ocp1}. Consider the constant $\alpha$ corresponding to Assumptions \ref{asuff2}, \ref{asuffstm} or \ref{asuffst}. Consider discrete optimal controls $\bar u_h\in \mathcal U_h$ of \eqref{discprob} that satisfy  $\|\bar u_h-\bar u\|_{L^1(\Omega_h)}< \alpha$. Further, assume that $\bar H_u $ satisfies the assumption of Lemma \ref{discretest}.
\begin{enumerate}
\item Let $L_b=0$ in the objective functional, let $\bar u$ be bang-bang and let $\bar u$ satisfy Assumption \ref{asuffst}$(\beta=1/2)$.
Then, there exists a positive constant $c$ independent of $h$ and a $h_0$ such that 
\begin{equation}\label{eqquation:improved1}
\|y_h(\bar u_h)-\bar y \|_{L^2(\Omega)}+{\|\varphi_h(\bar u_h)-\bar \varphi \|_{L^2(\Omega)}}\leq ch  \ \text{  for all } h\leq h_0.
\end{equation}
\item  Let {$\frac{\partial L_b}{\partial y}=0$} in the objective functional and let $\bar u$ satisfy Assumption \ref{asuffstm}$(\beta=1/2)$.
Then, there exists a positive constant $c$ independent of $h$ and a $h_0$ such that 
\begin{equation}\label{eqquation:improved2}
\|y_h(\bar u_h)-\bar y \|_{L^2(\Omega)}+{\|\varphi_h(\bar u_h)-\bar \varphi \|_{L^2(\Omega)}}\leq ch \ \text{  for all } h\leq h_0.
\end{equation}
\item 
Let $\bar u$ satisfy Assumption \ref{asuff2}$(\beta=1/2)$.
Then, there exists a positive constant $c$ independent of $h$ and a $h_0$ such that 
\begin{equation}\label{erlipcbb}
\|\bar u_h-\bar u\|_{L^1(\Omega_h)}+\|y_h(\bar u_h)-\bar y\|_{L^2(\Omega)}+{\|\varphi_h(\bar u_h)-\bar \varphi \|_{L^2(\Omega)}}\leq ch^{\frac{2\gamma}{\gamma+1}}   \ \text{  for all } h\leq h_0.
\end{equation}
\end{enumerate}
\end{Theorem}
\begin{proof}
  Most of the steps of the proof are the same as in the proof of Theorem \ref{pointest}. What is different is that instead of Lemma \ref{difdccc}, we apply Lemma \ref{lipest} together with Lemma \ref{discretest}, which allows for all bang-bang optimal controls $\bar u$, the estimate 
    \begin{equation}\label{equation:improvedest}
    \vert K_1+K_2 \vert =\vert J'(\bar u)(\Pi_h\bar u-\bar u)\vert \leq \mathrm{Lip}_{\bar \sigma}h\|\bar u-\Pi_h\bar u\|_{L^1(\Omega_h)}\leq c h^2.
    \end{equation}
    From here, we argue as before, to obtain the estimate 
    \begin{equation*}
    \kappa\|\bar u_h-\bar u\|_{L^1(\Omega_h)}^{1+\frac{1}{\gamma}}\leq J(\bar u_h)-J(\bar u) \leq c h^2\ \text{ for all } \bar u_h \text{ with }\|\bar u_h-\bar u\|_{L^1(\Omega_h)}<\alpha,
    \end{equation*}
    which yields following the same arguments as before the estimate \eqref{erlipcbb}.
    The claim under Assumption \ref{asuffstm}, \eqref{eqquation:improved2}, follows again using the same estimations as in the proof of Theorem \ref{pointest} together with the estimate \eqref{equation:improvedest}. Finally, \eqref{eqquation:improved1} is also a direct consequence of the estimations in the proof of Theorem \ref{pointest} and \eqref{equation:improvedest}.
    
\end{proof}

\subsection{Variational discretization}\label{vardisc}
We prove that Assumptions \ref{asuff2}, \ref{asuffstm}, and  \ref{asuffst} with $\beta =1$ are sufficient for error estimates for a variational discretization scheme. We refer to the \cite{H2005} for the idea and introduction of variational discretization. The assumptions on the objective functional we are considering are weaker than the ones in \cite{CM2021}, still the estimates given in Theorem \ref{vardiscest} below agree with the estimates in \cite[Remark 7]{CM2021} for the variational discretization.
We come to the error estimates for the variational discretization.
\begin{Theorem}
\label{vardiscest}
Let $\bar u$ be a local solution of \eqref{ocp1}. There exist positive constant $c$ and $h_0$ independent of $h$ such that for any sequence of solutions to the first-order optimality condition of the discrete problems, $\{\bar u_h\}_h$,  the following holds:
\begin{enumerate}
\item Let Assumption \ref{asuffst}$(\beta=1)$ be satisfied by $\bar u$. Then
\begin{align}\label{erestst}
\|y_h(\bar u_h)-\bar y\|_{L^2(\Omega)}+\|\varphi_h(\bar u_h)-\bar \varphi\|_{L^\infty(\Omega)}\leq ch   \ \text{  for all } h\leq h_0 .
\end{align}
\item  Let Assumption \ref{asuffstm}$(\beta=1)$ be satisfied by $\bar u$. Then
\begin{align}\label{ereststm}
\|y_h(\bar u_h)-\bar y\|_{L^2(\Omega)}+\|\varphi_h(\bar u_h)-\bar \varphi\|_{L^\infty(\Omega)}\leq c(h\vert\log h\vert)^2   \ \text{  for all } h\leq h_0.
\end{align}
\item Let Assumption \ref{asuff2}$(\beta=1)$ be satisfied by $\bar u$ for some $\gamma\in (n/(2+n),1]$. Then
\begin{align}\label{ereststall}
\|\bar u_h-\bar u\|_{L^1(\Omega_h)}+\|y_h(\bar u_h)-\bar y\|_{L^2(\Omega)}+\|\varphi_h(\bar u_h)-\bar \varphi\|_{L^\infty(\Omega)}\leq c(h\vert\log h\vert)^{2\gamma}   \ \text{  for all } h\leq h_0 .
\end{align}
\end{enumerate}
\end{Theorem}
\begin{proof}
We consider \eqref{ereststall}. Since $\bar u_h$ satisfies the first-order necessary optimality condition of the discrete problem, it holds
\begin{equation}\label{estimationfin}
\begin{aligned}
0&\geq J'_h(\bar u_h)(\bar u_h- \bar u)=J'(\bar u_h)(\bar u_h-\bar u)+J'_h(\bar u_h)(\bar u_h-u)-J'(\bar u_h)(\bar u_h-\bar u)\\
&\geq J'(\bar u)(\bar u_h-\bar u)+J''(\bar u)(\bar u_h-\bar u)^2-\vert J'(\bar u_h)(\bar u_h-\bar u)-J'(\bar u)(\bar u_h-\bar u)-J''(\bar u)(\bar u_h-\bar u)^2\vert\\
&+J'_h(\bar u_h)(\bar u_h-u)-J'(\bar u_h)(\bar u_h-\bar u).
\end{aligned}
\end{equation}
Utilizing Taylor's theorem, Lemma \ref{estfv} and Assumption \ref{asuff2}, we obtain
\begin{equation}\label{vardiscprep}
J'(\bar u_h)(\bar u_h-\bar u)-J'_h(\bar u_h)(\bar u_h-\bar u) \geq c\|\bar u_h-\bar u\|_{L^1(\Omega_h)}^{1+\frac{1}{\gamma}}.
\end{equation}
Since the discrete problem depends only on the values of the optimal control on the set $\Omega_h$, we may define $\bar u_h=\bar u$ on $\Omega\setminus\Omega_h$ and write 
\begin{align*}
&J'_h(\bar u_h)(\bar u_h-\bar u)-J'(\bar u_h)(\bar u_h-\bar u)=\int_{\Omega_h} (\varphi_h(\bar u_h)+L_b(x,y_h(\bar u_h)))(\bar u_h- \bar u)\dx\\
&-\int_{\Omega_h}(\varphi_{\bar u_h}+L_b(x,y_{\bar u_h}))(\bar u_h-\bar u)\dx=I.
\end{align*}
To estimate $I$, we follow similar reasoning as in \cite{CM2021}, using \eqref{estforstatelog}, Lemma \ref{discest}, Theorem \ref{interadj} and also using the local Lipschitz property of $L_b$ for $y$, to infer for some positive constant $c$.
\begin{align*}
&I\leq (\|\bar \varphi_h(\bar u_h)-\varphi_{\bar u_h}\|_{L^\infty(\Omega)}+\text{Lip}_{L_b,y;M}\|y_h(\bar u_h)-y_{\bar u_h}\|_{L^\infty(\Omega)})\|\bar u_h- \bar u\|_{L^1(\Omega_h)}\\
&\leq (\|\varphi_h(\bar u_h)-\varphi^h_{\bar u_h}\|_{L^\infty(\Omega)}+\| \varphi^h_{\bar u_h}-\varphi_{\bar u_h}\|_{L^\infty(\Omega)})\| \bar u_h- \bar u\|_{L^1(\Omega_h)}\\
&+\text{Lip}_{L_b,y;M}\|y_h(\bar u_h)-y_{\bar u_h}\|_{L^\infty(\Omega)}\| \bar u_h- \bar u\|_{L^1(\Omega_h)}\leq c(h^2+h^2\vert \log h\vert ^2)\| \bar u_h-\bar u\|_{L^1(\Omega_h)}.
\end{align*}
Altogether, utilizing \eqref{vardiscprep} we obtain for a positive constant again denoted by $c$ that
\begin{equation}\label{controlestv}
\|\bar u_h-\bar u\|_{L^1(\Omega_h)}\leq c(h^2+2h^2\vert \log h\vert ^2)^\gamma.
\end{equation}
For the states, we use \eqref{estforstatelog}, Lemma \ref{L2.2} and Lemma \ref{E2.12} to find for a positive constant $c$
\begin{align*}
    \|y_h(\bar u_h)-y_{\bar u}\|_{L^2(\Omega)}&\leq \|y_h(\bar u_h)-y_{\bar u_h}\|_{L^2(\Omega)}+\|y_{\bar u_h}-y_{\bar u} \|_{L^2(\Omega)}\leq c(h^2+\|\bar u_h-\bar u\|_{L^1(\Omega_h)})
\end{align*}
and the estimate follows from \eqref{controlestv}. The adjoints can be estimated from here using straightforward arguments.
Under Assumption \ref{asuffstm}, by \eqref{mixedvarcoer}, it holds
\begin{align*}
c\|\bar u_h-\bar u \|_{L^1(\Omega)}\|y_{\bar u_h}-y_{\bar u}\|_{L^2(\Omega)}\leq J'(\bar u_h)(\bar u_h-\bar u)-J'_h(\bar u_h)(\bar u_h-\bar u)\label{est46}.
\end{align*}
Estimating as before, we obtain the existence of a positive constant $c$ that satisfies
\begin{equation*}
\|y_{\bar u_h}-y_{\bar u}\|_{L^2(\Omega)}\leq c(h^2+h^2\vert \log h\vert ^2).
\end{equation*}
By again \eqref{estforstate2}, \eqref{distadj1} and \eqref{distadj2} the claim \eqref{ereststm} holds.
Finally, consider Assumption \ref{asuffst} and apply Lemma \ref{estfv2} to \eqref{estimationfin}, then it holds
\begin{align*}
c\|y_{\bar u_h}-y_{\bar u}\|_{L^2(\Omega)}^2\leq J'(\bar u_h)(\bar u_h-\bar u)-J'_h(\bar u_h)(\bar u_h-\bar u)\label{est46}.
\end{align*}
To estimate $I$, we use \eqref{distadj1}-\eqref{distadj2} to find
\begin{align*}
I&\leq \|\varphi_h(\bar u_h)-\varphi_{\bar u_h}\|_{L^2(\Omega)}\| \bar u_h- \bar u\|_{L^2(\Omega_h)}\\
&\leq (\|\varphi_h(\bar u_h)-\varphi^h_{\bar u_h}\|_{L^2(\Omega)}+\| \varphi^h_{\bar u_h}-\varphi_{\bar u_h}\|_{L^2(\Omega)})\| \bar u_h- \bar u\|_{L^2(\Omega_h)}\leq ch^2\|u_a-u_b\|_{L^\infty(\Omega)},
\end{align*}
for some positive constant $c$.
This leads to the estimate $\|y_{\bar u_h}-y_{\bar u}\|_{L^2(\Omega)}\leq ch$, and by \eqref{estforstate2} and \eqref{distadj1} the claim \eqref{erestst} holds.
\end{proof}
For a numerical example supporting the theoretical error estimates achieved in this paper, especially for the case $\gamma<1$, we refer to \cite{CM2021}.
\subsection{Solution stability and variational discretization}
One of the main results of this paper is that Assumptions \ref{asuff2}, \ref{asuffstm}, and \ref{asuffst} imply error estimates for the numerical approximation. These assumptions appeared first in the study of the solution stability of optimal control and states under perturbations appearing in the objective functional and the constraining PDE \cite{DJV2022, CDJ2023, DJV2022b}. We present an application of the solution stability property to obtain error estimates for a variational discretization scheme. In this sense, Theorem \ref{sifeer} below shows that a property related to solution stability guarantees the achievement of error estimates for a variational discretization scheme. The intuition we propose is that once solution stability is obtained under a growth condition on the objective functional, we can expect error estimates for the variational discretization scheme.

First, let us fix a positive constant $M$ and define the set of feasible perturbations by
\begin{equation*}
    \Big\{\zeta:=(\xi,\eta,\rho)\in L^{2}(\Omega)\times L^2(\Omega)\times L^\infty(\Omega)\vert \ \| \xi\|_{L^2(\Omega)}+\|\eta\|_{L^2(\Omega)}+\|\rho\|_{L^\infty(\Omega)}\leq M\Big\}.
\end{equation*}
Then, we define the perturbed problem by
\begin{equation}\label{perpp}
\min_{u\in \mathcal U} \Big\{J_\zeta(u):=\int_\Omega L(x,y_u, u)+\rho u+\eta y_u \dx\Big\}
\end{equation}
subject to \eqref{const} and
\begin{equation}
		\left\{ \begin{array}{lll}
		 Ay+f(\cdot,y)& =u+\xi\  &\text{ in }\ \Omega,\\
		y&=0\ &\text{ on } \partial \Omega.\\
	\end{array} \right.
	\label{sep1}
	\end{equation}
The existence of a globally optimal solution to \eqref{perpp}-\eqref{sep1} is guaranteed by the assumptions on the optimal control problem and the direct method in the calculus of variations.
Let us define a property that is strongly related to the notion of solution stability.
\begin{Definition}
\label{Dsmsrdr}
We call the optimal control problem \eqref{const}-\eqref{see1} to be solution stable at $\bar u$ for $V\subset \mathcal U$, with parameters $\kappa, \gamma$ and $\alpha$ if 
\begin{equation*}
    \|\bar u-\bar u^\zeta \|_{L^1(\Omega)}+\|\bar y-\bar y^\zeta\|_{L^2(\Omega)}+\|\bar \varphi-\bar \varphi^\zeta\|_{L^\infty(\Omega)} \leq \kappa \Big( \| \xi\|_{L^2(\Omega)} +\|\eta \|_{L^2(\Omega)}+\| \rho \|_{L^\infty(\Omega)}\Big)^\gamma
\end{equation*}
for all triples $(\bar u^\zeta, \bar y^\zeta,\bar \varphi^\zeta)$ corresponding to the perturbed problem \eqref{perpp}-\eqref{sep1} that satisfy

\noindent $\|\bar u-\bar u^\zeta\|_{L^1(\Omega)}<\alpha$ and $J_\zeta '(\bar u^\zeta)(v-\bar u^\zeta)\geq 0 \ \text{ for all } v\in V$.
\end{Definition}
Assumption \ref{asuff2} implies the notion of solution stability in Definition \ref{Dsmsrdr}. This can be observed by investigating the proof of the strong metric subregularity property of the optimality mapping in \cite{DJV2022b}.
\begin{Theorem}\label{sifeer}
    Let the optimal control problem be solution stable at $\bar u$ for $\{ \bar u\}$ with positive constants $\gamma, \kappa $ and $\alpha$. Let $\{\bar u_h\}_{h}$ be a sequence of solutions to the discrete problems \eqref{discprob} with $\|\bar {u_h}-\bar u\|_{L^1(\Omega)}<\alpha$. Then there exist positive constants $c(\kappa)$ and $h_0$ such that
    \begin{align*}
\|\bar u_h-\bar u\|_{L^1(\Omega_h)}+\|y_h(\bar u_h)-\bar y\|_{L^2(\Omega)}+\|\varphi_h(\bar u_h)-\bar \varphi\|_{L^\infty(\Omega)}\leq c(\kappa)(h\vert\log h\vert)^{2\gamma}\ \ \forall h\leq h_0.
\end{align*}
\end{Theorem}
\begin{proof}
  The idea is to construct a perturbed optimal control problem that relates the continuous problem with the discrete. This is done by considering a certain affine perturbation of the control, similar to the discussion in \cite[p. 4]{ktv2013}. Given a minimizer $\bar u_h$ of the discrete problem \eqref{discprob}, let $\zeta:=(0,0,\rho)$, with $\rho:=\varphi_h(\bar u_h)-\varphi_{\bar u_h}$. Then, we define the perturbed optimal control problem
    \begin{equation*}
 		  \min_{u \in \Uad}\bigg\{ J_\zeta(u) := \int_\Omega L(x,y(x),u(x))\dx+\int_\Omega (\varphi_h(\bar u_h)-\varphi_{\bar u_h})u\dx \bigg\},
	\end{equation*}
subject to \eqref{see1}.
It is easy to see that $J_\zeta'(\bar u_h)(\bar u-\bar u_h)=\int_\Omega \varphi_h(\bar u_h)(\bar u-\bar u_h)\dx\geq 0$.
But that is all we need of $\bar u_h$ to apply the solution stability at $\bar u$. That is, we obtain $\|\bar u_h-\bar u\|_{L^1(\Omega_h)}\leq \kappa\|(\varphi(\bar u_h)-\varphi_{\bar u_h}) \|_{L^\infty(\Omega)}^\gamma. $ By Theorem \ref{estlincont}, Lemma \ref{L2.2}, Lemma \ref{discest} and Theorem \ref{varthmest} the claim follows.
\end{proof}

\subsection*{Acknowledgement}
The author is thankful to E. Casas, A. Dom\'{i}nguez Corella, V. M. Veliov, anonymous readers, and the anonymous referees for many helpful comments and suggestions that improved the quality of the manuscript.

\bibliographystyle{plain}
\bibliography{bibe}
\end{document}